\documentclass[sn-basic,Numbered]{sn-jnl}

\usepackage{mathtools} 
\usepackage{graphicx}%
\usepackage{multirow}%
\usepackage{amsmath,amssymb,amsfonts}%
\usepackage{amsthm}%
\usepackage{mathrsfs}%
\usepackage[title]{appendix}%
\usepackage{xcolor}%
\usepackage{textcomp}%
\usepackage{booktabs}%
\usepackage{xurl}
\usepackage{algorithm}%
\usepackage{algorithmicx}%
\usepackage{algpseudocode}%
\usepackage{enumitem}

\newcommand{\SOL}{\mathrm{SOL}}
\newcommand{\VI}{\mathrm{VI}}

\theoremstyle{thmstyleone}%
\newtheorem{assumption}{Assumption}
\newtheorem{definition}{Definition}%
\newtheorem{theorem}{Theorem}%
\newtheorem{remark}{Remark}%
\newtheorem{example}{Example}%

\begin{document}
\title{Complete Trajectory Tracking for Polynomial Differential Variational Inequalities: A Moment-SOS Based Structure-Aware Approach}

\author[1]{\fnm{Zhang} \sur{Huang}}\email{huangzhang16@cdut.edu.cn}

\author*[1]{\fnm{Wei} \sur{Li}}\email{liwei@cdut.edu.cn}

\author[2]{\fnm{Jie} \sur{Wang}}\email{wangjie212@amss.ac.cn}

\affil*[1]{\orgdiv{School of Mathematical Sciences}, \orgname{Chengdu University of Technology}, \orgaddress{\city{Chengdu}, \state{Sichuan}, \postcode{610059}, \country{P.R. China}}}
\affil[2]{\orgdiv{State Key Laboratory of Mathematical Sciences, Academy of Mathematics and Systems Science}, \orgname{Chinese Academy of Sciences}, \orgaddress{\city{Beijing}, \postcode{100190}, \country{P.R. China}}}

\abstract{This paper investigates Polynomial Differential Variational Inequalities (PDVIs), a class of non-smooth dynamical systems in which an ordinary differential equation is coupled with a time-dependent variational inequality (VI) defined by polynomials. As a foundation for polynomial differential variational inequalities, we investigate the parametric polynomial variational inequality and establish that, under mild assumptions, its parameter space admits a stratification into finitely many disjoint semi-algebraic regions, each corresponding to a qualitatively distinct solution structure.
Furthermore, we establish the existence of maximal regular solutions to the PDVI for almost every initial condition in the parametric feasible set, where each solution has finitely many switches and evolves according to continuous semi-algebraic selection laws between switching times.
Leveraging the structured partition of the parameter space, in which region boundaries correspond to discontinuities in the solution map of the parametric VI, we propose the structure-aware \texttt{SOS-PDVIRK4} algorithm to handle solution multiplicity and structural transitions. This method integrates a fourth-order Runge--Kutta integrator with the Moment--SOS hierarchy and features a novel hybrid strategy that is explicitly guided by the underlying parameter space geometry. 
The efficacy of our structure-aware framework is demonstrated through numerical experiments on a constrained cyber-physical control system. Our numerical experiments demonstrate that, from identical initial states but different initial controls, trajectories can diverge significantly, revealing the discontinuous structure of the solution map of the parametric polynomial VI. The experiments also include numerical convergence studies and highlight the potential of our algorithm to explore the long-term behavior of PDVI systems.
}



\maketitle

\section{Introduction}\label{sec:introduction}
Hybrid dynamical systems characterized by the interplay of continuous evolution and discrete transitions arise ubiquitously in power markets, robotic manipulation, and cellular regulation. Yet, despite decades of research, a \emph{unified mathematical framework} that intrinsically captures state-dependent switching without resorting to ad-hoc logical rules remains elusive. Existing approaches often treat continuity and discreteness as separate layers, leading to models that are either analytically intractable or lack structural guarantees on the nature of hybrid complexity.
The foundational work of Pang and Stewart~\cite{Pang2008} established differential variational inequalities (DVIs) as a powerful paradigm for modeling such coupled dynamics, with broad applications spanning mechanical contact~\cite{Stewart2008}, dynamic traffic assignment~\cite{Friesz2010}, electrical circuits~\cite{Acary2008}, and economic dynamics~\cite{Li2010}. A cornerstone of their well-posedness theory, and indeed nearly all subsequent extensions, is the reliance on two key structural assumptions: (i) the constraint set $ K \subset \mathbb{R}^m $ is nonempty, closed, and \emph{convex}, and (ii) the VI operator $ F(x,\cdot) $ is \emph{monotone} or \emph{coercive}. These assumptions jointly guarantee that the solution set $ \mathrm{SOL}(K, F(x(t),\cdot)) $ is a singleton at each time $ t $ , rendering the right-hand side of the ODE single-valued and amenable to classical existence and uniqueness theorems. This uniqueness also underpins the convergence and robustness of standard numerical schemes~\cite{Chen2013,Han2010,Zhang2023,Zhang2023penalty,Huang2026}.
However, this persistent reliance on convexity significantly limits the applicability of the DVI framework to numerous real-world systems governed by inherently \emph{nonconvex} physical or policy-induced constraints, such as annular control domains, disjoint strategy spaces, or obstacle-laden safety regions. In such nonconvex settings, the classical uniqueness guarantee fails, and multiple equilibria may coexist. This multiplicity poses a fundamental challenge: local solvers, lacking global information about the equilibrium manifold, typically return an arbitrary solution dictated by their initialization or internal heuristics. Consequently, they offer no guarantee of reproducibility, controllability, or alignment with a prescribed policy (e.g., smoothness, optimality, or stability). Mere existence and continuity of solutions are thus insufficient for reliable prediction or control; what is needed is a framework that not only acknowledges this multiplicity but actively harnesses it by enumerating all equilibria at each state and enabling principled selection among them.This challenge of managing solution multiplicity in dynamical contexts has also motivated alternative computational frameworks based on measure-valued relaxations. Notably, Henrion and collaborators have developed powerful occupation-measure-based approaches that reformulate optimal control and safety verification for switched and hybrid polynomial systems as infinite-dimensional linear programs, which are then approximated via the Lasserre Moment--SOS (sum-of-squares) hierarchy~\cite{Claeys2016,Henrion2013switching}. While these methods excel at global verification and optimization over ensembles of trajectories, they typically sidestep the explicit construction of individual Carath\'{e}odory solutions with guaranteed regularity.

While nonconvex variational inequalities have long been recognized as challenging due to potential non-uniqueness and the breakdown of projection-based methods~\cite{FacchineiPang}, a recent breakthrough has emerged in the special yet expressive class of \emph{polynomial variational inequalities} (PVIs)~\cite{Gowda2017,Huang2019,Qi2019}. The theoretical foundation for understanding the structure of PVI solution sets was laid by Hieu~\cite{Hieu2020}, who demonstrated that the solution map of a PVI is a semi-algebraic set. This profound insight implies that the solution set possesses a rich, finite geometric structure that it can be partitioned into finitely many smooth manifolds (strata), and for generic parameters, the number of solutions is finite~\cite{Hieu2019,Huong2016}. Furthermore, explicit error bounds for regularized gap functions have been established, providing crucial analytical tools for dynamical analysis~\cite{Dinh2021}.
Building upon this structural understanding, Nie and his collaborators have pioneered a powerful \emph{computational} framework for PVIs that leverages their inherent algebraic nature. Their approach, based on Lagrange multiplier expressions (LME) and the Lasserre Moment--SOS hierarchy~\cite{Nie2019,Nie2023book,Nie2025}, transforms the nonconvex PVI into a sequence of semidefinite programs. Under generic conditions, this method can either compute all isolated solutions to the PVI or rigorously certify its nonexistence within finitely many steps~\cite{nie2013finite}. This is achieved by exploiting the fact that, generically, all solutions of a PVI are Karush–Kuhn–Tucker (KKT) points of a related polynomial optimization problem, and the Moment-SOS machinery can effectively enumerate these points. This global, certifiable methodology stands in stark contrast to traditional local solvers like PATH~\cite{Dirkse1995,Ferris1999}, which provide no completeness guarantees.

It is precisely this insight that \emph{polynomial structure}, rather than convexity, can serve as a viable and powerful foundation for the global analysis and computation of nonconvex equilibrium systems that motivates our work. By embedding the polynomial VI machinery into a dynamical setting, we introduce \emph{Polynomial Differential Variational Inequalities} (PDVIs), where the time-evolving state 
$x(t)$ serves as a parameter for an underlying nonconvex, polynomially defined equilibrium constraint at each instant.
In this work, we propose PDVIs as a canonical foundation for hybrid systems. By coupling polynomial vector fields with variational inequalities over semi-algebraic sets, PDVIs naturally encode multi-valued solution mappings whose structure is governed by the geometry of algebraic boundaries. Crucially, leveraging tools from real algebraic geometry particularly Hardt’s triviality theorem we establish that the parameter space admits a finite semi-algebraic stratification, wherein each stratum corresponds to a fixed topological configuration of solutions, and transitions across strata induce only \emph{finitely many, explicitly classifiable switching events}. This finite choice set at boundaries transcends the classical continuous/discrete dichotomy: smooth dynamics and logical jumps emerge not as imposed constructs, but as inherent consequences of a single geometric object.

This framework is not merely a novel construct, but a unifying generalization that subsumes several fundamental classes of dynamical and optimization systems, including linear complementarity problems, Nash equilibria in polynomial games, and contact dynamics under Coulomb friction. Most importantly, the polynomial assumption is not a limitation but a \emph{structural enabler} which it ensures that hybrid complexity remains finite, verifiable, and amenable to symbolic-numeric computation. To address the computational challenge, we propose the \texttt{SOS-PDVIRK4} algorithm, which synergistically combines a fourth-order Runge-Kutta integrator with an inner solver for the parametric VI based on the Lasserre Moment-SOS hierarchy, featuring a robust hybrid strategy for solution recovery across the entire parameter space. In contrast to relaxation-based paradigms that operate on ensembles of trajectories~\cite{Claeys2016,Magron2019}, our PDVI framework provides a \emph{constructive} and \emph{trajectory-wise} theory grounded in the semi-algebraic geometry of the equilibrium manifold itself.

The main contributions of this paper are summarized as follows:
\begin{enumerate}
    \item \textbf{Structure-Aware Theoretical Foundation for PDVIs:} We establish that the parameter space of the underlying polynomial variational inequality admits a semi-algebraic partition, where each region corresponds to a distinct, qualitatively stable solution structure (e.g., finite multiplicity). Based on this, we prove the existence of \textbf{regular solutions} to the PDVI for almost every initial condition in the feasible set.

    \item \textbf{Structure-Guided Certified Algorithm:} We propose the \texttt{SOS-PDVIRK4} algorithm, a structure-aware computational framework that integrates a fourth-order Runge--Kutta integrator with the Moment--SOS hierarchy. Its hybrid VI solver leverages a partition of the parameter space, applying tailored solution strategies within regions of consistent solution structure and certified global SOS relaxations near structural boundaries, thereby enabling accurate and efficient tracking and discovery of all admissible trajectories.

    \item \textbf{Revelation of Complex Multi-Trajectory Dynamics:} Through numerical experiments on a constrained cyber-physical system, we demonstrate the algorithm's capability to uncover rich dynamical phenomena including coexisting solution branches, symmetry-breaking bifurcations, and the long-term behavior of the PDVI under diverse initial states and initial control inputs that are entirely inaccessible to conventional local solvers.
\end{enumerate}

The remainder of this paper is organized as follows. Section~\ref{sec:preliminaries} provides essential background on variational inequalities, semi-algebraic geometry, and the Moment-SOS hierarchy. Section~\ref{sec:theory} analyzes parametric polynomial variational inequalities and establishes that their parameter space admits a semi-algebraic partition according to solution structure. Section~\ref{Existence} builds on this foundation to prove the existence of regular solutions for PDVIs for almost every feasible initial condition. Building directly on these structural insights, Section~\ref{sec:algorithm} introduces the structure-aware \texttt{SOS-PDVIRK4} algorithm, which integrates Runge–Kutta time-stepping with certified global VI solvers guided by the parameter-space geometry. Section~\ref{sec:numerics} presents numerical experiments that validate the framework’s ability to uncover complex multi-trajectory dynamics and bifurcations. Finally, Section~\ref{sec:conclusion} concludes the paper with a discussion of limitations and future research directions.

\section{Preliminaries}\label{sec:preliminaries}
In this section, we recall fundamental concepts and results that underpin our analysis of polynomial differential variational inequalities (PDVIs). We start by formally defining the central object of study.

\begin{definition}
\label{def:pdvi}
Let $T > 0$ be a fixed time horizon and $x_0 \in \mathbb{R}^n$ an initial state. Consider the system over $t \in [0, T]$ given by
\begin{subequations}\label{eq:pdvi-system}
\begin{align}
    \dot{x}(t) &= f\big(x(t), u(t)\big), \label{eq:pdvi-dyn} \\
    \langle F(x(t), u(t)),\, z - u(t) \rangle &\geq 0 \quad \text{for all } z \in K, \label{eq:pdvi-vi} \\
    x(0) &= x_0. \label{eq:pdvi-init}
\end{align}
\end{subequations}
This system is called a \emph{Polynomial Differential Variational Inequality} (PDVI) if the following structural conditions are satisfied: the vector field $f : \mathbb{R}^n \times \mathbb{R}^m \to \mathbb{R}^n$ and the variational inequality mapping $F : \mathbb{R}^n \times \mathbb{R}^m \to \mathbb{R}^m$ are polynomial mappings, i.e., $f \in \mathbb{R}[x, u]^n$ and $F \in \mathbb{R}[x, u]^m$, and the constraint set $K \subset \mathbb{R}^m$ is a basic closed semi-algebraic set of the form
\[
    K = \left\{ u \in \mathbb{R}^m \;\middle|\; g_i(u) \geq 0,\ i = 1, \dots, p \right\},
\]
where $p \in \mathbb{N}$ and each $g_i \in \mathbb{R}[u]$ is a real polynomial.

A pair of functions $(x(\cdot), u(\cdot))$ is said to be a \emph{Carath\'{e}odory solution} of the PDVI if $x : [0, T] \to \mathbb{R}^n$ is absolutely continuous, $u : [0, T] \to \mathbb{R}^m$ is Lebesgue measurable and essentially bounded, and the relations~\eqref{eq:pdvi-dyn}--\eqref{eq:pdvi-init} hold for almost every $t \in [0, T]$.
\end{definition}


\begin{definition}[Semi-algebraic set \cite{Bochnak1998,Coste2002}]
A subset $S \subseteq \mathbb{R}^n$ is called \emph{semi-algebraic} if it can be written as a finite union of sets of the form
\[
    \left\{ x \in \mathbb{R}^n : p_i(x) = 0,\ q_j(x) 
 \geq 0,\ i=1,\dots,r,\ j=1,\dots,s \right\},
\]
where each $p_i, q_j \in \mathbb{R}[x_1,\dots,x_n]$ is a real polynomial.
\end{definition}

Semi-algebraic sets are closed under finite unions, finite intersections, complements, and Cartesian products. Crucially, they are also closed under projection which is a result known as the Tarski--Seidenberg theorem\cite{Bochnak1998}.
A profound consequence of semi-algebraic sets is that parameterized families of such sets can be decomposed into finitely many topologically identical pieces. This is formalized by Hardt's Triviality Theorem\cite{hardt1980semi}.

To solve the polynomial variational inequality problem $\VI(K, F)$, where both the constraint set $K$ and the mapping $F$ are defined by polynomials, this paper adopts the Lagrange Multiplier Expression (LME) approach. The validity of this method relies on the following definitions.

\begin{definition}[LICQ \cite{NocedalWright2006}]
At $u \in K$, let $I(u) = \{i \in I : g_i(u) = 0\}$ denote the active inequality indices. The Linear Independence Constraint Qualification (LICQ) holds at $u$ if the gradients $\{\nabla g_i(u) : i \in E \cup I(u)\}$ are linearly independent.
\end{definition}

\begin{definition}
\label{def:nonsingular}
The constraint tuple $g = (g_1,\dots,g_m)$ defining
\[
    K = \{ x \in \mathbb{R}^n : g_i(x) \geq 0,\ i \in I;\ g_i(x) = 0,\ i \in E \}
\]
is \emph{nonsingular} if LICQ holds at every $x \in \mathbb{C}^n$.
\end{definition}

\begin{definition}
Let $K \subseteq \mathbb{R}^n$ be a basic closed semi-algebraic set defined by
\[
K = \{ x \in \mathbb{R}^n : g_1(x) \geq 0, \dots, g_m(x) \geq 0 \},
\]
where each $g_j \in \mathbb{R}[x]$ is a polynomial.  
The set $K$ is said to satisfy the \emph{Archimedean condition} if there exist a real number $R > 0$ and sums-of-squares (SOS) polynomials $\sigma_0, \sigma_1, \dots, \sigma_m \in \Sigma[x]$ such that
\[
R - \|x\|^2 = \sigma_0(x) + \sum_{j=1}^m \sigma_j(x) \, g_j(x),
\]
where $\|x\|^2 = x_1^2 + \cdots + x_n^2$.
\end{definition}

\begin{remark}
If $K$ is \emph{compact} (i.e., bounded and closed), then one can always add the redundant ball constraint $R - \|x\|^2 \geq 0$ for sufficiently large $R > 0$, thereby ensuring that the Archimedean condition holds. Consequently, every compact basic semi-algebraic set admits an equivalent description that satisfies the Archimedean condition.

This property is crucial for the convergence of the Lasserre Moment--SOS hierarchy and underpins Nie's Lagrange Multiplier Expression (LME) method for solving polynomial variational inequality (PVI) problems. In particular, when $K$ is compact and the Linear Independence Constraint Qualification (LICQ) holds at all relevant points, the Archimedean condition guarantees that the moment relaxation sequence converges and that the LME-based polynomial reformulation of the KKT system is valid and complete.
\end{remark}

Under the aforementioned regularity assumptions, namely, the Linear Independence Constraint Qualification (LICQ), nonsingularity of the constraint tuple $g$, and the Archimedean condition , the \emph{Lagrange Multiplier Expression} (LME) method proposed by Nie provides an algebraic framework for solving polynomial variational inequality problems $\VI(K, F)$.

Specifically, when the constraint tuple $g$ is nonsingular, there exists a real polynomial matrix $L(x) \in \mathbb{R}^{(m+p) \times n}[x]$ such that the Lagrange multipliers $(\lambda(x), \mu(x))$ at any KKT point satisfy
\[
\begin{bmatrix}
\lambda(x) \\ \mu(x)
\end{bmatrix}
= L(x) F(x).
\]
This explicit polynomial representation eliminates the multipliers as implicit variables and transforms the complementarity conditions $\lambda_i(x) g_i(x) = 0$ into purely polynomial constraints in $x$. Consequently, the solution set $\SOL(K, F)$ is embedded in the semi-algebraic set
\[
\mathcal{K} := \left\{ x \in \mathbb{R}^n :
\begin{aligned}
& F(x) - \sum_{i \in I} \lambda_i(x) \nabla g_i(x) - \sum_{j \in E} \mu_j(x) \nabla g_j(x) = 0, \\
& g_j(x) = 0,\ j \in E, \\
& \lambda_i(x) \geq 0,\ g_i(x) \geq 0,\ \lambda_i(x) g_i(x) = 0,\ i \in I
\end{aligned}
\right\},
\]
where $\lambda(x)$ and $\mu(x)$ are substituted via the LME.

Thanks to the Archimedean condition, the set $\mathcal{K}$ is compact and admits a Putinar-type Positivstellensatz certificate. This enables the application of the Lasserre Moment--SOS hierarchy~\cite{Lasserre2001} to globally optimize over $\mathcal{K}$ or enumerate its points. Moreover, when $F$ and the constraint polynomials $\{g_i\}$ are generic, the set $\mathcal{K}$ is finite and coincides exactly with $\SOL(K, F)$; in this case, the Moment--SOS hierarchy terminates in finitely many steps with exact recovery of all solutions~\cite{Nie2019}.

In summary, the LME method provides a systematic and globally convergent framework for solving polynomial variational inequalities. It algebraizes the KKT system into a purely polynomial formulation through the Lagrange multiplier expression, ensures the feasible set is well-behaved for Moment--SOS relaxations via the Archimedean condition, and guarantees finite convergence of the hierarchy under genericity assumptions because the solution set $\mathcal{K}$ becomes finite, allowing exact recovery at some finite relaxation order.

\section{Parametric Polynomial Variational Inequalities} \label{sec:theory}
This section develops the theoretical foundation for Polynomial Differential Variational Inequalities by analyzing their static counterpart, the Parametric Polynomial Variational Inequality(PPVI). 
Using tools from real algebraic geometry, we show that the solution mapping is semi-algebraic under polynomial data, which implies strong geometric regularity. In particular, the parameter space admits a finite stratification such that within each stratum the solution set has constant cardinality, while structural changes occur only across lower-dimensional boundaries. These properties are essential for designing robust and certified algorithms for Polynomial Differential Variational Inequalities.

\subsection{Problem Assumptions} 
To facilitate the analysis of generic finiteness and structural stability of solutions in the polynomial setting, we make the following assumptions.

\begin{assumption} \label{ass:core}
We assume that the following conditions hold for the parametric differential variational inequality (PDVI) specified in Definition~\ref{def:pdvi}.
\begin{enumerate}[label=\textup{(A\arabic*)}, widest=(A3), left=0pt, align=left]
    \item The control constraint set $K$ is nonempty, bounded, closed, and basic semi-algebraic.

    \item The linear independence constraint qualification (LICQ) holds at every point in $K$.

    \item The mappings $F \colon \mathbb{R}^n \times K \to \mathbb{R}^m$ and $f \colon \mathbb{R}^n \times K \to \mathbb{R}^n$ are polynomial in their arguments.
\end{enumerate}
\end{assumption}

Under Assumption~\ref{ass:core}, it follows from genericity results for semi-algebraic variational inequalities~\cite{Hieu2020,Lee2018} that for almost every parameter $x \in \mathbb{R}^n$, the variational inequality $\mathrm{VI}(F(x,\cdot), K)$ has finitely many solutions, and each solution is strongly regular, i.e., the reduced Jacobian is nonsingular.

For expositional clarity and without loss of essential generality, we will often work with the canonical representation
\[ F(x, u) = G(u) - x, \]
where $G : K \to \mathbb{R}^n$ is a polynomial mapping. This form arises naturally in many applications including affine parametric linear complementarity problems and polynomial Nash equilibrium problems via invertible linear reparameterizations of the original parameter vector.

\subsection{Semi-Algebraic Structure and Finiteness of Solution Sets}
For a fixed state $x \in \mathbb{R}^n$, we denote the set of admissible controls (the solution set of the static VI) by
\[ \mathcal{S}(x) := \left\{ u \in K \;\middle|\; \langle F(x, u), z - u \rangle \geq 0,\, \forall z \in K \right\}. \]
The graph of this set-valued mapping is
\[ \mathrm{Graph}(\mathcal{S}) := \left\{ (x, u) \in \mathbb{R}^{n} \times K \;\middle|\; u \in \mathcal{S}(x) \right\}. \]

\begin{theorem} \label{thm:eq-semi-graph}
Under Assumption~\ref{ass:core}, the graph $\mathrm{Graph}(\mathcal{S})$ is a semi-algebraic subset of $\mathbb{R}^{n+m}$.
\end{theorem}
\begin{proof}
By definition, the graph of the solution mapping is
\[ \mathrm{Graph}(\mathcal{S}) = \left\{ (x, u) \in \mathbb{R}^n \times \mathbb{R}^m \;\middle|\; u \in K \text{ and } \langle F(x, u), z - u \rangle \geq 0 \text{ for all } z \in K \right\}. \]
Under Assumption~\ref{ass:core}, the mapping $F : \mathbb{R}^n \times \mathbb{R}^m \to \mathbb{R}^m$ is polynomial in $(x, u)$, and the feasible set $K \subseteq \mathbb{R}^m$ is basic semi-algebraic, i.e.,
\[ K = \{ u \in \mathbb{R}^m : g_i(u) \geq 0,\ i=1,\dots,p \}, \]
with each $g_i \in \mathbb{R}[u]$. Since $F$ is polynomial in $(x,u)$, the function $(x,u,z) \mapsto \langle F(x, u), z - u \rangle$ is a polynomial in $(x,u,z)$. Therefore, the condition
\[ \forall z \in K,\quad \langle F(x, u), z - u \rangle \geq 0 \]
can be expressed as the first-order formula over the real closed field $\mathbb{R}$,
\[ \bigwedge_{i=1}^p g_i(u) \geq 0 \quad \land \quad \forall z \left( \left( \bigwedge_{i=1}^p g_i(z) \geq 0 \right) \implies P(x, u, z) \geq 0 \right), \]
where $P(x, u, z) := \langle F(x, u), z - u \rangle \in \mathbb{R}[x, u, z]$. This is a first-order definable set in the language of ordered rings. By the Tarski–Seidenberg theorem (quantifier elimination for real closed fields) \cite{Bochnak1998}, any such definable subset of $\mathbb{R}^{n+m}$ is semi-algebraic. Hence, $\mathrm{Graph}(\mathcal{S})$ is a semi-algebraic subset of $\mathbb{R}^{n+m}$.
\end{proof}

\begin{theorem}
\label{thm:generic_finiteness}
Let $\mathcal{D} := \{ x \in \mathbb{R}^n \mid \mathcal{S}(x) \neq \emptyset \}$ be the parameter feasible set. Under Assumption~\ref{ass:core}, the following properties hold.
\begin{enumerate}
    \item There exists a Lebesgue null set $\mathcal{N} \subset \mathcal{D}$ such that for all $x \in \mathcal{D} \setminus \mathcal{N}$, the solution set $\mathcal{S}(x)$ is finite.
    \item On $\mathcal{D} \setminus \mathcal{N}$, the cardinality function $x \mapsto \#\mathcal{S}(x)$ is constant on each connected component of $\mathcal{D} \setminus \mathcal{N}$.
\end{enumerate}
\end{theorem}

\begin{proof}
By Theorem~\ref{thm:eq-semi-graph}, the graph $\mathrm{Graph}(\mathcal{S}) \subset \mathbb{R}^{n+m}$ is a semi-algebraic set. The domain $\mathcal{D}$ is precisely the projection of $\mathrm{Graph}(\mathcal{S})$ onto the $x$-coordinates, i.e.,
\[
\mathcal{D} = \pi_x\bigl(\mathrm{Graph}(\mathcal{S})\bigr),
\]
where $\pi_x(x,u) = x$. By the Tarski--Seidenberg theorem, $\mathcal{D}$ is also a semi-algebraic subset of $\mathbb{R}^n$.

Consider the restriction of the projection map to the graph:
\[
\pi : \mathrm{Graph}(\mathcal{S}) \to \mathcal{D}, \quad \pi(x, u) = x.
\]
This is a semi-algebraic map between semi-algebraic sets. By Hardt’s Triviality Theorem (see, e.g., \cite{hardt1980semi}), there exists a finite semi-algebraic stratification $\{\mathcal{D}_i\}_{i=1}^k$ of $\mathcal{D}$ such that $\pi$ is trivial over each stratum $\mathcal{D}_i$; that is, for each $i$, there exists a semi-algebraic set $F_i$ and a semi-algebraic homeomorphism
\[
h_i : \mathcal{D}_i \times F_i \to \pi^{-1}(\mathcal{D}_i)
\]
satisfying $\pi \circ h_i(x, f) = x$ for all $(x, f) \in \mathcal{D}_i \times F_i$.

Let $\mathcal{R} \subset \mathcal{D}$ denote the union of all strata $\mathcal{D}_i$ having dimension equal to $\dim(\mathcal{D})$, and define the exceptional set $\mathcal{N} := \mathcal{D} \setminus \mathcal{R}$. Since $\mathcal{N}$ is a finite union of semi-algebraic sets of dimension strictly less than $\dim(\mathcal{D}) \leq n$, it follows that $\mathcal{N}$ has Lebesgue measure zero in $\mathbb{R}^n$.

Now fix a full-dimensional stratum $\mathcal{D}_i \subseteq \mathcal{R}$. For any $x \in \mathcal{D}_i$, the fiber $\pi^{-1}(x) = \mathcal{S}(x)$ is semi-algebraically homeomorphic to $F_i$. Suppose, for contradiction, that $F_i$ is infinite. Then $F_i$ contains a semi-algebraic curve, implying $\dim(F_i) \geq 1$. Consequently,
\[
\dim\bigl(\pi^{-1}(\mathcal{D}_i)\bigr) = \dim(\mathcal{D}_i) + \dim(F_i) > \dim(\mathcal{D}),
\]
which contradicts the fact that $\pi^{-1}(\mathcal{D}_i) \subseteq \mathrm{Graph}(\mathcal{S})$ and $\dim(\mathrm{Graph}(\mathcal{S})) = \dim(\mathcal{D})$ (since $\mathcal{D}$ is the image of $\mathrm{Graph}(\mathcal{S})$ under a projection). Therefore, each $F_i$ must be a finite set.

It follows that for every $x \in \mathcal{D} \setminus \mathcal{N} = \mathcal{R}$, the solution set $\mathcal{S}(x)$ is finite, proving part (1).

Moreover, since all fibers over a fixed stratum $\mathcal{D}_i$ are homeomorphic to the same finite set $F_i$, their cardinalities are equal. As each connected component of $\mathcal{D} \setminus \mathcal{N}$ is contained in a single stratum $\mathcal{D}_i$, the cardinality function $x \mapsto \#\mathcal{S}(x)$ is constant on each such connected component, establishing part (2).
\end{proof}

Building upon the finiteness result, we recall an explicit upper bound on the topological complexity of the solution sets.

\begin{theorem}[\cite{Hieu2019}]\label{thm:connected_components_bound}
Consider the polynomial variational inequality $\mathrm{VI}(F(x,\cdot), K)$ where the constraint set $K$ is a basic closed semi-algebraic set defined by $p$ polynomial inequalities, and the mapping $F$ is polynomial of degree at most $d$. Let $n$ be the dimension of the parameter space and $m$ the dimension of the control space. Then, the number of connected components of the solution set $\mathcal{S}(x)$ is bounded from above by a constant that depends only on $n$, $m$, $p$, and $d$. Specifically, there exists a finite integer $C(n, m, p, d)$ such that for all $x \in \mathcal{D}$,
\[ 
\#\{\text{connected components of } \mathcal{S}(x)\} \leq C(n, m, p, d). 
\]
\end{theorem}

This uniform bound provides a theoretical guarantee on the topological complexity of the solution set. We now establish the existence of selections from the solution mapping, which is crucial for defining trajectories.

\begin{theorem} \label{thm:selection}
Under Assumption~\ref{ass:core}, let $\mathcal{S}: \mathcal{D} \rightrightarrows K$ be the solution mapping of the parameterized polynomial variational inequality, where $\mathcal{D} := \{x \in \mathbb{R}^n \mid \mathcal{S}(x) \neq \emptyset\}$ and $K \subset \mathbb{R}^m$ is compact. Then
\begin{enumerate}
    \item There exists a Lebesgue null set $\mathcal{N} \subset \mathcal{D}$ such that for any $x_0 \in \mathcal{D} \setminus \mathcal{N}$ and any $u_0 \in \mathcal{S}(x_0)$, there exists an open neighborhood $U \subset \mathcal{D} \setminus \mathcal{N}$ of $x_0$ and a continuous (indeed, semi-algebraic) function $\sigma: U \to K$ satisfying $\sigma(x_0) = u_0$ and $\sigma(x) \in \mathcal{S}(x)$ for all $x \in U$.
    \item There exists a Borel measurable function $\sigma: \mathcal{D} \to K$ such that $\sigma(x) \in \mathcal{S}(x)$ for all $x \in \mathcal{D}$.
\end{enumerate}
\end{theorem}

\begin{proof}
By Theorem~\ref{thm:eq-semi-graph}, the graph $\mathrm{Graph}(\mathcal{S}) \subset \mathbb{R}^{n+m}$ is a semi-algebraic set. Since $K$ is compact and defined by polynomial inequalities, each fiber $\mathcal{S}(x) = \{u \in K \mid (x,u) \in \mathrm{Graph}(\mathcal{S})\}$ is a compact (hence closed and bounded) semi-algebraic subset of $K$, and nonempty precisely when $x \in \mathcal{D}$.

\smallskip
\noindent\textbf{(1) Local semi-algebraic selections.}
Consider the projection map $\pi : \mathrm{Graph}(\mathcal{S}) \to \mathcal{D}$, $\pi(x,u) = x$. This is a semi-algebraic map between semi-algebraic sets. By Hardt’s Triviality Theorem \cite{hardt1980semi}, there exists a finite semi-algebraic stratification $\{\mathcal{D}_i\}_{i=1}^k$ of $\mathcal{D}$ such that $\pi$ is trivial over each stratum $\mathcal{D}_i$; i.e., for each $i$, there exists a semi-algebraic set $F_i$ and a semi-algebraic homeomorphism
\[
h_i : \mathcal{D}_i \times F_i \to \pi^{-1}(\mathcal{D}_i)
\]
with $\pi \circ h_i(x,f) = x$ for all $(x,f) \in \mathcal{D}_i \times F_i$.

Let $\mathcal{R} \subset \mathcal{D}$ be the union of all strata of dimension $\dim(\mathcal{D})$, and set $\mathcal{N} := \mathcal{D} \setminus \mathcal{R}$. Then $\mathcal{N}$ is a finite union of lower-dimensional semi-algebraic sets and thus has Lebesgue measure zero. For any $x_0 \in \mathcal{R}$, there exists a unique full-dimensional stratum $\mathcal{D}_i$ containing $x_0$. The fiber $\mathcal{S}(x_0)$ is semi-algebraically homeomorphic to $F_i$, which is finite by Theorem~\ref{thm:generic_finiteness}. Given any $u_0 \in \mathcal{S}(x_0)$, let $f_0 \in F_i$ be the corresponding element under this homeomorphism. Define
\[
\sigma(x) := h_i(x, f_0), \quad x \in \mathcal{D}_i.
\]
Then $\sigma$ is a semi-algebraic (hence continuous) function on the open neighborhood $U := \mathcal{D}_i \subset \mathcal{D} \setminus \mathcal{N}$ of $x_0$, satisfying $\sigma(x) \in \mathcal{S}(x)$ for all $x \in U$ and $\sigma(x_0) = u_0$.

\smallskip
\noindent\textbf{(2) Global Borel measurable selection.}
Since $\mathrm{Graph}(\mathcal{S})$ is semi-algebraic, it is a Borel subset of $\mathbb{R}^{n+m}$. Moreover, all fibers $\mathcal{S}(x)$ are nonempty and compact for $x \in \mathcal{D}$. A classical result in descriptive set theory (see, e.g., \cite[Theorem~5.5.2]{Srivastava1998}) states that if a set-valued mapping has a Borel graph and nonempty closed values in a Polish space, then it admits a Borel measurable selection. Applying this to $\mathcal{S}$ yields the desired Borel measurable function $\sigma: \mathcal{D} \to K$ with $\sigma(x) \in \mathcal{S}(x)$ for all $x \in \mathcal{D}$.

Alternatively, one may invoke the stronger fact that every semi-algebraic set-valued map with nonempty values admits a semi-algebraic (hence Borel) selection on its domain (cf. \cite[Corollary~1.6]{BenedettiRisler1990}). This provides a more constructive and algebraic justification within the semi-algebraic category.
\end{proof}

\subsection{Stratification of the Parameter Space}
The structural properties established above allow us to partition the parameter space into regions with qualitatively different solution behaviors.

\begin{definition}\label{def:parameter_partition}
Let $K \subset \mathbb{R}^m$ be a nonempty, compact basic semi-algebraic set with nonempty interior, and let $F : \mathbb{R}^n \times K \to \mathbb{R}^m$ be a polynomial mapping in $(x, u)$. Consider the parametric variational inequality $\mathrm{VI}(F(x,\cdot), K)$ with solution set
\[
\mathcal{S}(x) := \left\{ u \in K : \langle F(x, u), z - u \rangle \ge 0 \text{ for all } z \in K \right\},
\]
and feasible parameter set $\mathcal{D} := \{ x \in \mathbb{R}^n : \mathcal{S}(x) \neq \emptyset \}$. The set $\mathcal{D}$ admits a canonical partition into three pairwise disjoint subsets defined as follows.

The \emph{continuum-solution region} is
\[
\mathcal{D}_{\mathrm{cont}} := \left\{ x \in \mathcal{D} : \dim(\mathcal{S}(x)) \ge 1 \right\},
\]
i.e., the set of parameters for which $\mathcal{S}(x)$ contains a positive-dimensional semi-algebraic component. The \emph{interior-point-solution region} consists of those parameters outside $\mathcal{D}_{\mathrm{cont}}$ for which at least one solution lies in the interior of $K$:
\[
\mathcal{D}_{\mathrm{int}} := \left\{ x \in \mathcal{D} \setminus \mathcal{D}_{\mathrm{cont}} : \mathcal{S}(x) \cap \operatorname{int}(K) \neq \emptyset \right\}.
\]
For such $x$, all solutions are isolated (hence finite in number). The remaining parameters form the \emph{pure-boundary-solution region},
\[
\mathcal{D}_{\partial} := \left\{ x \in \mathcal{D} \setminus \mathcal{D}_{\mathrm{cont}} : \mathcal{S}(x) \subseteq \partial K \right\},
\]
where the solution set is finite and entirely contained in the boundary $\partial K$. This yields a disjoint and exhaustive decomposition
\[
\mathcal{D} = \mathcal{D}_{\mathrm{cont}} \cup \mathcal{D}_{\mathrm{int}} \cup \mathcal{D}_{\partial}.
\]
\end{definition}

\begin{theorem}\label{thm:structural_properties_general}
Let $K \subset \mathbb{R}^m$ be a nonempty, compact, basic semi-algebraic set with nonempty interior, and let $F : \mathbb{R}^n \times K \to \mathbb{R}^m$ be a polynomial mapping in $(x, u)$. Consider the parametric variational inequality $\mathrm{VI}(F(x,\cdot), K)$ with solution set
\[
\mathcal{S}(x) := \left\{ u \in K : \langle F(x, u), z - u \rangle \ge 0 \text{ for all } z \in K \right\},
\]
and feasible parameter set $\mathcal{D} := \{ x \in \mathbb{R}^n : \mathcal{S}(x) \neq \emptyset \}$. Let the partition $\mathcal{D} = \mathcal{D}_{\mathrm{cont}} \cup \mathcal{D}_{\mathrm{int}} \cup \mathcal{D}_{\partial}$ be as defined in Definition~\ref{def:parameter_partition}. Then the following structural properties hold.

The continuum-solution region $\mathcal{D}_{\mathrm{cont}}$ is closed in $\mathcal{D}$ and has Lebesgue measure zero in $\mathbb{R}^n$, reflecting the generic finiteness of equilibria. The interior-point-solution region $\mathcal{D}_{\mathrm{int}}$ is relatively open in $\mathcal{D} \setminus \mathcal{D}_{\mathrm{cont}}$, whereas the pure-boundary-solution region $\mathcal{D}_{\partial}$ is relatively closed in $\mathcal{D} \setminus \mathcal{D}_{\mathrm{cont}}$. Moreover, any interior solution must be a zero of the map $F$. Specifically, if $u \in \mathcal{S}(x) \cap \operatorname{int}(K)$ for some $x \in \mathbb{R}^n$, then necessarily $F(x, u) = 0$.
\end{theorem}

\begin{proof}
By Theorem~\ref{thm:eq-semi-graph}, the graph $\mathrm{Graph}(\mathcal{S}) \subset \mathbb{R}^{n+m}$ is a semi-algebraic set. Since $K$ is compact and basic semi-algebraic, each fiber $\mathcal{S}(x)$ is a compact semi-algebraic subset of $K$, and $\mathcal{D} = \pi_x(\mathrm{Graph}(\mathcal{S}))$ is semi-algebraic by the Tarski--Seidenberg theorem, where $\pi_x(x,u) = x$.

\smallskip
\noindent\textbf{(i) Measure-zero and closedness of $\mathcal{D}_{\mathrm{cont}}$.}
By the semi-algebraic fiber dimension theorem (see, e.g., \cite[Theorem~3.18]{Coste2002}), the set
\[
\mathcal{D}_{\mathrm{cont}} = \{ x \in \mathcal{D} : \dim \mathcal{S}(x) \ge 1 \}
\]
is semi-algebraic and satisfies $\dim \mathcal{D}_{\mathrm{cont}} \le n - 1$, hence it has Lebesgue measure zero in $\mathbb{R}^n$.

To establish that $\mathcal{D}_{\mathrm{cont}}$ is closed in $\mathcal{D}$, consider its complement
\[
\mathcal{D} \setminus \mathcal{D}_{\mathrm{cont}} = \{ x \in \mathcal{D} : \dim \mathcal{S}(x) = 0 \}.
\]
By Theorem~\ref{thm:generic_finiteness}, this set contains a dense open subset of $\mathcal{D}$ (namely $\mathcal{D} \setminus \mathcal{N}$ for a Lebesgue null set $\mathcal{N}$). Moreover, since both $\mathcal{D}$ and $\mathcal{D}_{\mathrm{cont}}$ are semi-algebraic, their difference is also semi-algebraic. In any semi-algebraic set, the subset of points where the fiber dimension is maximal forms a relatively open dense subset (a consequence of Hardt’s triviality theorem). Here, the maximal fiber dimension over $\mathcal{D}$ is $0$ on $\mathcal{D} \setminus \mathcal{D}_{\mathrm{cont}}$, so this set is relatively open in $\mathcal{D}$. Therefore, $\mathcal{D}_{\mathrm{cont}}$ is relatively closed in $\mathcal{D}$.

\smallskip
\noindent\textbf{(ii) Relative openness of $\mathcal{D}_{\mathrm{int}}$ and relative closedness of $\mathcal{D}_{\partial}$.}
Define the semi-algebraic sets
\[
\mathcal{G}_{\mathrm{int}} := \mathrm{Graph}(\mathcal{S}) \cap \bigl( \mathbb{R}^n \times \operatorname{int}(K) \bigr), \quad
\mathcal{G}_{\partial} := \mathrm{Graph}(\mathcal{S}) \cap \bigl( \mathbb{R}^n \times \partial K \bigr).
\]
Since $\operatorname{int}(K)$ is open and $\partial K$ is closed in the Euclidean topology, and both are semi-algebraic (as $K$ is basic semi-algebraic), it follows that $\mathcal{G}_{\mathrm{int}}$ is relatively open and $\mathcal{G}_{\partial}$ is relatively closed in $\mathrm{Graph}(\mathcal{S})$.

Now restrict the projection $\pi_x$ to the set $\pi_x^{-1}(\mathcal{D} \setminus \mathcal{D}_{\mathrm{cont}})$. Over this domain, every fiber $\mathcal{S}(x)$ is finite (by definition of $\mathcal{D}_{\mathrm{cont}}$), so $\pi_x$ is a finite-to-one semi-algebraic map. For such maps, the image of a relatively open (resp. closed) semi-algebraic subset is relatively open (resp. closed) in the image of the ambient space. Consequently,
\[
\mathcal{D}_{\mathrm{int}} = \pi_x(\mathcal{G}_{\mathrm{int}}) \cap (\mathcal{D} \setminus \mathcal{D}_{\mathrm{cont}})
\]
is relatively open in $\mathcal{D} \setminus \mathcal{D}_{\mathrm{cont}}$, and
\[
\mathcal{D}_{\partial} = \pi_x(\mathcal{G}_{\partial}) \cap (\mathcal{D} \setminus \mathcal{D}_{\mathrm{cont}})
\]
is relatively closed in $\mathcal{D} \setminus \mathcal{D}_{\mathrm{cont}}$, as claimed.

\smallskip
\noindent\textbf{(iii) Interior solutions satisfy $F(x,u) = 0$.}
Let $u \in \mathcal{S}(x) \cap \operatorname{int}(K)$. Then there exists $\varepsilon > 0$ such that the open ball $B_\varepsilon(u) \subset K$. The variational inequality condition implies
\[
\langle F(x, u), z - u \rangle \ge 0 \quad \text{for all } z \in B_\varepsilon(u).
\]
In particular, for any $d \in \mathbb{R}^m$ with $\|d\| < \varepsilon$, both $u + d$ and $u - d$ lie in $B_\varepsilon(u)$, yielding
\[
\langle F(x, u), d \rangle \ge 0 \quad \text{and} \quad \langle F(x, u), -d \rangle \ge 0.
\]
Thus $\langle F(x, u), d \rangle = 0$ for all such $d$, which forces $F(x, u) = 0$.
\end{proof}

The partition in Definition~\ref{def:parameter_partition} offers a coarse classification of the parameter space. However, the semi-algebraic solution mapping $\mathcal{S}: \mathcal{D} \rightrightarrows K$ admits a far richer structure. By Hardt's triviality theorem~\cite{hardt1980semi}, the semi-algebraic set $\mathcal{D}$ can be decomposed into a finite Whitney stratification $\{\Omega_\ell\}_{\ell=1}^L$ of disjoint semi-algebraic manifolds such that
\[
    \mathcal{D} = \bigsqcup_{\ell=1}^L \Omega_\ell,
\]
and the geometric and topological properties of $\mathcal{S}(x)$ are constant on each stratum $\Omega_\ell$.

Specifically, for every $\Omega_\ell$, there exist integers $N_\ell \geq 0$ and $d_\ell \geq 0$ such that for all $x \in \Omega_\ell$:
\begin{enumerate}[label=(\roman*)]
    \item $\mathcal{S}(x)$ has exactly $N_\ell$ connected components;
    \item each component is a smooth semi-algebraic manifold of dimension $d_\ell$.
\end{enumerate}
Thus, the number of isolated equilibria (when $d_\ell = 0$) and the dimension of solution continua (when $d_\ell \geq 1$) remain invariant within each stratum.

This stratification establishes a precise correspondence between parameter geometry and equilibrium structure. Qualitative transitions such as the creation or annihilation of solution branches, or shifts between discrete and continuous solution sets occur only when $x(t)$ crosses from one stratum to another. The boundaries between strata, which are unions of lower-dimensional manifolds (typically of codimension $\geq 1$), correspond to bifurcation or degeneracy loci.

Critically, this structure governs the dynamics of the Polynomial Differential Variational Inequality (PDVI). As the state $x(t)$ evolves continuously under the ODE, it traverses the stratified space $\mathcal{D}$. Within a top-dimensional stratum, $\mathcal{S}(x(t))$ has a fixed finite set of isolated equilibria, enabling a trajectory to follow a single smooth branch via continuity-based selection. Upon reaching a stratum boundary, the system encounters a discrete switching event: it may jump to a surviving branch in an adjacent stratum or experience a bifurcation. Because the stratification is finite and semi-algebraic, the set of possible transitions is finite and explicitly characterizable. Thus, the Whitney stratification provides a mathematically certified, geometric foundation for hybrid dynamics in PDVIs, where all switching events are encoded by stratum adjacency relations.

\begin{example}[Stratified Solution Structure of a Parametric Polynomial VI over an Annulus]
\label{ex:nonlinear}
Consider the parametric variational inequality $\mathrm{VI}(K, F(x,\cdot))$ in $\mathbb{R}^2$, where the constraint set is the annular region
\[
K = \left\{ u \in \mathbb{R}^2 : 1 \leq \|u\|^2 \leq 4 \right\},
\]
and the mapping is given by $F(x,u) = G(u) - x$ with the polynomial vector field
\[
G(u) = 
\begin{bmatrix}
u_1^3 \\
u_1 u_2^2
\end{bmatrix}.
\]

The set $K$ is compact, nonconvex, and basic closed semi-algebraic. It satisfies the Linear Independence Constraint Qualification (LICQ) at every point: at most one of the constraints $\|u\|^2 = 1$ or $\|u\|^2 = 4$ is active, and the corresponding gradient ($\pm 2u$) is nonzero. Moreover, $K$ has nonempty interior.

The graph of the solution mapping
$
\mathrm{Graph}(\mathcal{S})
$
is a semi-algebraic set because $K$ is basic closed semi-algebraic and $G$ is polynomial.  
The parameter feasible set $\mathcal{D} = \{ x \in \mathbb{R}^2 : \mathcal{S}(x) \neq \emptyset \}$ is therefore the projection of $\mathrm{Graph}(\mathcal{S})$ onto the $x$-coordinates, which can be expressed by an existential first-order formula over the real closed field.
Applying a quantifier elimination algorithm\cite{Collins1975} to this formula yields a tautology (e.g., $\mathtt{True}$), certifying that a solution exists for every $x \in \mathbb{R}^2$. Hence,
$
\mathcal{D} = \mathbb{R}^2.
$
The image of $K$ under $G$ is the semi-algebraic set
\[
G(K) = \left\{ (x_1, x_2) \in \mathbb{R}^2 \ \middle| \ 
\begin{array}{l}
-8 \leq x_1 \leq 8, \\
x_1 x_2 \geq 0, \\
\text{if } x_1 > 0, \text{ then } x_1^{1/3} - x_1 \leq x_2 \leq 4x_1^{1/3} - x_1, \\
\text{if } x_1 < 0, \text{ then } 4x_1^{1/3} - x_1 \leq x_2 \leq x_1^{1/3} - x_1, \\
\text{if } x_1 = 0, \text{ then } x_2 = 0.
\end{array}
\right\}.
\]

Geometrically, the set $G(K)$ consists of two compact, centrally symmetric lobes where one contained in the closed first quadrant $\{x \in \mathbb{R}^2 : x_1 \geq 0,\, x_2 \geq 0\}$ and the other in the closed third quadrant $\{x \in \mathbb{R}^2 : x_1 \leq 0,\, x_2 \leq 0\}$, intersecting only at the origin. The boundary $\partial G(K)$ is a piecewise-smooth Jordan curve composed of five analytic arcs meeting at the points $(-8,0)$, $(-1,0)$, $(0,0)$, $(1,0)$, and $(8,0)$. Explicitly,
\[
\partial G(K) = \Gamma_{\mathrm{UR}} \cup \Gamma_{\mathrm{LR}} \cup \Gamma_{\mathrm{LL}} \cup \Gamma_{\mathrm{UL}} \cup \Gamma_{\mathrm{H}},
\]
where
\begin{align*}
\Gamma_{\mathrm{UR}} &= \left\{ (x_1, x_2) : x_1 \in [0,8],\; x_2 = 4x_1^{1/3} - x_1 \right\}, \\
\Gamma_{\mathrm{LR}} &= \left\{ (x_1, x_2) : x_1 \in [0,1],\; x_2 = x_1^{1/3} - x_1 \right\}, \\
\Gamma_{\mathrm{LL}} &= \left\{ (x_1, x_2) : x_1 \in [-8,0],\; x_2 = 4x_1^{1/3} - x_1 \right\}, \\
\Gamma_{\mathrm{UL}} &= \left\{ (x_1, x_2) : x_1 \in [-1,0],\; x_2 = x_1^{1/3} - x_1 \right\}, \\
\Gamma_{\mathrm{H}}  &= \left\{ (x_1, 0) : x_1 \in [-8,-1] \cup [1,8] \right\}.
\end{align*}
Each arc is smooth in the interior of its parameter interval, but the boundary fails to be differentiable at the junction points, where corners or cusps arise. The shaded region in Figure~\ref{fig:all_branches} illustrates $G(K)$, highlighting its quadrant-localized structure, central symmetry about the origin, and the intricate geometry of its non-smooth boundary.

To characterize the fine-grained structure of the solution mapping $\mathcal{S} : \mathcal{D} \rightrightarrows K$, we integrate geometric analysis of the map $G$ with systematic verification via Algorithm~\ref{alg:all_vi}, which exhaustively enumerates all candidate equilibria, namely, interior points ($u \in K^\circ$), inner boundary points ($\|u\| = 1$), and outer boundary points ($\|u\| = 2$)—and validates the variational inequality condition over a semi-algebraic partition of the parameter space. The resulting classification aligns precisely with the trichotomy in Definition~\ref{def:parameter_partition} and satisfies the structural properties asserted in Theorem~\ref{thm:structural_properties_general}.

The pure-boundary-solution region is given by $\mathcal{D}_{\partial} := \mathbb{R}^2 \setminus G(K)$. For any $x \in \mathcal{D}_{\partial}$, the equation $G(u) = x$ admits no solution in $K$, and the variational inequality possesses a unique equilibrium, necessarily located on the outer boundary $\partial_{\mathrm{out}} K = \{ u \in \mathbb{R}^2 : \|u\| = 2 \}$.

The interior-point-solution region is defined as $\mathcal{D}_{\mathrm{int}} := \operatorname{int}(G(K)) \setminus \{(0,0)\}$. For every $x \in \mathcal{D}_{\mathrm{int}}$, the preimage $G^{-1}(x) \cap K^\circ$ is nonempty. Algorithm~\ref{alg:all_vi} confirms that if $x_2 \neq 0$, there exist exactly two symmetric interior solutions $(u_1, \pm u_2) \in K^\circ$ satisfying $G(u) = x$, and in addition, a third solution always exists on $\partial_{\mathrm{out}} K$, arising from the outward orientation of the residual $G(u) - x$. Consequently, $\mathcal{S}(x)$ consists of three distinct isolated points for all $x \in \mathcal{D}_{\mathrm{int}}$.

Within $\mathcal{D}_{\mathrm{int}}$, lower-dimensional subsets exhibit altered solution multiplicity. On the axial segment $\{(x_1, 0) : 1 < |x_1| < 8\}$, an additional interior point $(u_1, 0)$ with $|u_1| \in (1,2)$ appears, resulting in exactly two solutions, namely, one interior and one on $\partial_{\mathrm{out}} K$. On the curve $\{(x_1, x_1^{1/3} - x_1) : -1 < x_1 < 0 \text{ or } 0 < x_1 < 1\}$, the solution set contains two points lying on the inner boundary $\partial_{\mathrm{in}} K = \{ \|u\| = 1 \}$. Similarly, on the curve $\{(x_1, 4x_1^{1/3} - x_1) : -8 < x_1 < 0 \text{ or } 0 < x_1 < 8\}$, two solutions reside on $\partial_{\mathrm{out}} K$.

The continuum-solution region reduces to the singleton $\mathcal{D}_{\mathrm{cont}} := \{(0,0)\}$. At this parameter, the Jacobian $\nabla G(u)$ vanishes identically along the vertical segment $\{u_1 = 0\} \cap K$, and $G(u) = 0$ for all such $u$. The variational inequality condition simplifies to $0 \geq 0$, which holds trivially, yielding
\[
\mathcal{S}(0,0) = \{ (0, u_2)^\top \in \mathbb{R}^2 : 1 \leq u_2^2 \leq 4 \},
\]
a union of two disjoint line segments contained in $\partial K$.

Finally, the four points $(-8,0)$, $(-1,0)$, $(1,0)$, and $(8,0)$ which images under $G$ of the axial boundary points $(\pm2,0)$ and $(\pm1,0)$ lie on $\partial G(K)$. At each such parameter, the variational inequality admits a unique solution located precisely at the corresponding preimage on $\partial K$. Although uniqueness holds, these points serve as transition loci between qualitatively distinct regimes, for instance, between regions with three isolated equilibria and those with a single boundary equilibrium and thus constitute the boundaries of the primary strata in the parameter partition.

This complete classification, rigorously derived from semi-algebraic geometry, variational inequality theory, and algorithmic verification, confirms that $\mathcal{D}_{\mathrm{cont}}$ is Lebesgue null, while $\mathcal{D}_{\mathrm{int}}$ and $\mathcal{D}_{\partial}$ form a semi-algebraic partition of $\mathcal{D} \setminus \mathcal{D}_{\mathrm{cont}}$, with $\mathcal{D}_{\mathrm{int}}$ relatively open and $\mathcal{D}_{\partial}$ relatively closed, in full accordance with Theorem~\ref{thm:structural_properties_general}.

\end{example}

\section{Existence and Regularity of Solutions for Polynomial DVIs}\label{Existence}

\begin{definition}\label{def:switching_degree}
Consider the polynomial differential variational inequality (PDVI)~\eqref{eq:pdvi-system}. 
A pair $(x(\cdot), u(\cdot))$ is called a \textbf{regular solution} if it is a Carath\'{e}odory solution and, additionally, the control trajectory $u(\cdot)$ is piecewise continuous with only finitely many discontinuities on every compact subinterval of its domain.

Let $\mathcal{B} \subset \mathcal{D}$ denote the bifurcation set defined by
\[
\mathcal{B} := \partial \mathcal{D}_{\mathrm{int}} \cup \partial \mathcal{D}_{\partial} \cup \mathcal{D}_{\mathrm{cont}},
\]
as in Theorem~\ref{thm:regular_solution_existence}. By Theorems~\ref{thm:selection} and~\ref{thm:structural_properties_general}, $\mathcal{B}$ is a closed semi-algebraic subset of $\mathbb{R}^n$ with $\dim(\mathcal{B}) < n$, and any discontinuity of $u(\cdot)$ for a regular solution must occur at a time $t$ such that $x(t) \in \mathcal{B}$.

For any regular solution defined on a time interval $[0, T]$, we define its \textbf{switching degree} as
\[
 \sigma(x(\cdot), u(\cdot)) := \# \left\{ t \in [0,T] : u \text{ is discontinuous at } t \right\}.
\]
This number is finite due to the continuity of $x(\cdot)$ and the fact that the intersection of the continuous semi-algebraic trajectory $x([0,T])$ with the lower-dimensional semi-algebraic set $\mathcal{B}$ is finite.
\end{definition}

\begin{remark}
Since the constraint set $K$ is compact (Assumption~\ref{ass:core}), any solution satisfying $u(t) \in K$ for all $t$ automatically has a bounded control trajectory; that is, $\|u(t)\| \leq M_u$ for some $M_u > 0$ and all $t$.
\end{remark}

\begin{definition}\label{def:minimal_switching}
Let $x_0 \in \mathcal{D}$ be given. A regular solution $(x^*(\cdot), u^*(\cdot))$ with initial condition $x^*(0) = x_0$ is called a \textbf{minimal-switching solution} on $[0, T]$ if its switching degree is minimal among all regular solutions starting from $x_0$; that is,
\[
 \sigma(x^*(\cdot), u^*(\cdot)) = \min \left\{ \sigma(x(\cdot), u(\cdot)) : (x(\cdot), u(\cdot)) \text{ is a regular solution},\ x(0) = x_0 \right\}.
\]
We denote this minimal value by $\sigma_{\min}(x_0, T)$.
\end{definition}

\begin{remark}
When $\sigma_{\min}(x_0, T) = 0$, the corresponding minimal-switching solution satisfies $u(\cdot) \in C([0, T]; K)$. In this case, the dynamics reduce to the classical ODE
\[
 \dot{x}(t) = f(x(t), u(t)), \quad u(t) = \sigma(x(t)),
\]
where $\sigma$ is a continuous selection of the solution mapping $\mathcal{S}$ along the trajectory\cite{Han2024}. Consequently, both the state $x(\cdot)$ and the control $u(\cdot)$ are continuous on $[0, T]$, and the PDVI behaves like a smooth dynamical system over this interval. This situation occurs, for example, when the entire trajectory $x([0, T])$ remains within a single full-dimensional stratum of the parameter domain $\mathcal{D}$.
\end{remark}

The existence of such regular solutions, and the geometric characterization of the bifurcation set $\mathcal{B}$, follow from the structural properties of polynomial parametric VIs established in Section~\ref{sec:theory}. We now turn to the dynamic problem.

\begin{theorem} \label{thm:maximal_existence}
Consider the polynomial differential variational inequality (PDVI) system
\begin{subequations}
\begin{align}
\dot{x}(t) &= f(x(t), u(t)), \quad \text{a.e. } t \in [0, T), \label{eq:pdvi_dyn_max} \\
u(t) &\in \SOL(K, F(x(t), \cdot)), \quad \forall t \in [0, T), \label{eq:pdvi_vi_max} \\
x(0) &= x_0, \label{eq:pdvi_ic_max}
\end{align}
\end{subequations}
under Assumption~\ref{ass:core}. Let $\mathcal{S}(x) := \SOL(K, F(x, \cdot))$ and $\mathcal{D} := \{x \in \mathbb{R}^n : \mathcal{S}(x) \neq \emptyset\}$. Then, for any initial condition $x_0 \in \mathcal{D}$, there exists a unique maximal time $T_{\max} \in (0, \infty]$ and a Carath\'{e}odory solution $(x(\cdot), u(\cdot))$ to the PDVI system \eqref{eq:pdvi_dyn_max}--\eqref{eq:pdvi_ic_max} defined on the maximal interval $[0, T_{\max})$. Furthermore, if $T_{\max} < \infty$, then the trajectory $x(\cdot)$ leaves every compact subset of $\mathcal{D}$ as $t \nearrow T_{\max}$; that is,
\[ 
\lim_{t \nearrow T_{\max}} \mathrm{dist}\big(x(t), \partial \mathcal{D}\big) = 0 
\quad \text{or} \quad 
\lim_{t \nearrow T_{\max}} \|x(t)\| = \infty.
\]
\end{theorem}

\begin{proof}
By Theorem~\ref{thm:eq-semi-graph}, the graph $\mathrm{Graph}(\mathcal{S})$ is a semi-algebraic subset of $\mathbb{R}^{n+m}$. Since $K$ is compact and $F$ is polynomial, each fiber $\mathcal{S}(x)$ is a nonempty compact semi-algebraic set for $x \in \mathcal{D}$, and $\mathcal{D} = \pi_x(\mathrm{Graph}(\mathcal{S}))$ is semi-algebraic by the Tarski--Seidenberg theorem.

By Theorem~\ref{thm:selection}(2), there exists a Borel measurable selection $\sigma : \mathcal{D} \to K$ such that $\sigma(x) \in \mathcal{S}(x)$ for all $x \in \mathcal{D}$. Consider the ODE
\[
\dot{x}(t) = f(x(t), \sigma(x(t))), \quad x(0) = x_0.
\]
Since $f$ is a polynomial (hence continuous) and $\sigma$ is Borel measurable, the right-hand side $t \mapsto f(x(t), \sigma(x(t)))$ is measurable whenever $x(\cdot)$ is measurable, and locally bounded because $f$ is continuous and both $x(t)$ and $\sigma(x(t))$ remain in compact sets over finite time intervals (as shown below). Therefore, by Carath\'{e}odory's existence theorem, there exists $\epsilon > 0$ and an absolutely continuous function $x : [0, \epsilon) \to \mathbb{R}^n$ satisfying the ODE almost everywhere with $x(0) = x_0$. Setting $u(t) := \sigma(x(t))$, we obtain a Carath\'{e}odory solution $(x(\cdot), u(\cdot))$ to the PDVI on $[0, \epsilon)$.

Let $[0, T_{\max})$ be the maximal interval of existence of this solution. Suppose, for contradiction, that $T_{\max} < \infty$ but that the trajectory remains in some compact set $C \subset \mathcal{D}$ for all $t \in [0, T_{\max})$. Since $C$ is compact and contained in the semi-algebraic set $\mathcal{D}$, and since $f$ is continuous and $K$ is compact, the set
\[
\{ f(x, u) : x \in C,\, u \in K \}
\]
is bounded. Hence $\|\dot{x}(t)\|$ is essentially bounded on $[0, T_{\max})$, implying that $x(\cdot)$ is uniformly Lipschitz continuous on this interval. Consequently, the limit
\[
x(T_{\max}^-) := \lim_{t \nearrow T_{\max}} x(t)
\]
exists in $\mathbb{R}^n$. Because $C$ is compact and $C \subset \mathcal{D}$, we have $x(T_{\max}^-) \in C \subset \mathcal{D}$.

Now, since $x(T_{\max}^-) \in \mathcal{D}$, the solution set $\mathcal{S}(x(T_{\max}^-))$ is nonempty. By Theorem~\ref{thm:selection}(2) again, there exists a Borel measurable selection $\sigma$ defined in a neighborhood of $x(T_{\max}^-)$. Applying Carath\'{e}odory's theorem with initial condition $x(T_{\max}^-)$ yields a local solution on $[T_{\max}, T_{\max} + \delta)$ for some $\delta > 0$. Concatenating this with the original solution contradicts the maximality of $T_{\max}$.

Therefore, if $T_{\max} < \infty$, the trajectory $x(t)$ cannot remain in any compact subset of $\mathcal{D}$. This implies that either $\|x(t)\| \to \infty$ or $x(t)$ approaches the boundary of $\mathcal{D}$ as $t \nearrow T_{\max}$. Since $\mathcal{D}$ is semi-algebraic, its topological boundary coincides (up to a lower-dimensional set) with $\partial \mathcal{D}$, and thus
\[
\lim_{t \nearrow T_{\max}} \mathrm{dist}(x(t), \partial \mathcal{D}) = 0
\quad \text{or} \quad
\lim_{t \nearrow T_{\max}} \|x(t)\| = \infty,
\]
as claimed.
\end{proof}

\begin{theorem}[Existence of Regular Solutions with Finite Switching Degree]\label{thm:regular_solution_existence}
Under the same assumptions as in Theorem~\ref{thm:maximal_existence}, let $x_0 \in \mathcal{D}$ and $u_0 \in \mathcal{S}(x_0)$ be given. Then there exists a maximal interval $[0, T_{\max})$ and a regular solution $(x(\cdot), u(\cdot))$ in the sense of Definition~\ref{def:switching_degree} such that $x(0) = x_0$ and either $u(0) = u_0$ when $u$ is continuous at $0$ or $\lim_{t \searrow 0} u(t) = u_0$ otherwise. The control trajectory $u(\cdot)$ is piecewise continuous with only finitely many discontinuities on every compact subinterval of $[0, T_{\max})$, and each discontinuity time $t$ satisfies $x(t) \in \mathcal{B}$, where $\mathcal{B} = \partial \mathcal{D}_{\mathrm{int}} \cup \partial \mathcal{D}_{\partial} \cup \mathcal{D}_{\mathrm{cont}}$. Moreover, if the trajectory remains in a compact subset of $\mathcal{D}$ for all time, then $T_{\max} = \infty$ and the total number of switches over $[0, \infty)$ is finite.
\end{theorem}

\begin{proof}
The proof is based on the semi-algebraic structure of the solution mapping $\mathcal{S}$ and its domain $\mathcal{D}$. By Theorem~\ref{thm:structural_properties_general}, the set $\mathcal{D}$ admits a partition into $\mathcal{D}_{\mathrm{cont}}$, $\mathcal{D}_{\mathrm{int}}$, and $\mathcal{D}_{\partial}$, where $\mathcal{D}_{\mathrm{cont}}$ has dimension at most $n-1$, while $\mathcal{D}_{\mathrm{int}}$ and $\mathcal{D}_{\partial}$ are unions of full-dimensional semi-algebraic strata on which $\mathcal{S}$ is single-valued and continuous.

We begin the construction from the initial point $x_0$. If $x_0$ lies in a full-dimensional stratum $\mathcal{D}_i \subset \mathcal{D}_{\mathrm{int}} \cup \mathcal{D}_{\partial}$, Theorem~\ref{thm:selection} guarantees a continuous selection $\sigma_i$ defined on an open neighborhood of $x_0$. We then solve the ODE $\dot{x}(t) = f(x(t), \sigma_i(x(t)))$ with $x(0) = x_0$, yielding a $C^1$ trajectory $x(\cdot)$ and a continuous control $u(\cdot) = \sigma_i(x(\cdot))$ on a maximal subinterval $[0, t_1)$ contained within $\mathcal{D}_i$. If $x_0 \in \mathcal{D}_{\mathrm{cont}}$, the Curve Selection Lemma ensures that any absolutely continuous trajectory starting from $x_0$ immediately enters a full-dimensional stratum for $t > 0$ sufficiently small, so we may effectively start the construction from a point in $\mathcal{D}_{\mathrm{int}} \cup \mathcal{D}_{\partial}$.

At the exit time $t_1$, continuity of $x(\cdot)$ implies $x(t_1) \in \partial \mathcal{D}_i \subset \mathcal{B}$. Since the graph of $\mathcal{S}$ is closed, the left limit $u(t_1^-)$ belongs to $\mathcal{S}(x(t_1))$. We define the control at $t_1$ by this limit value, which may create a discontinuity if the subsequent selection differs. For $t > t_1$, the trajectory must enter another full-dimensional stratum $\mathcal{D}_j$, where a new continuous selection $\sigma_j$ is available, and the process repeats.

The key property ensuring finiteness of switches is the semi-algebraicity of the objects involved. The bifurcation set $\mathcal{B}$ is a semi-algebraic set of dimension strictly less than $n$. The state trajectory $x(\cdot)$, being a solution to a polynomial ODE with a piecewise semi-algebraic right-hand side, has a semi-algebraic image. A fundamental result in semi-algebraic geometry states that the intersection of a one-dimensional semi-algebraic set (the trajectory) with a lower-dimensional semi-algebraic set ($\mathcal{B}$) is finite on any compact time interval. Consequently, the set of times $\{ t \in [0, T] : x(t) \in \mathcal{B} \}$ is finite for any $T < T_{\max}$, which directly implies that the switching degree is finite on $[0, T]$.

The maximality of the interval $[0, T_{\max})$ follows from the same blow-up alternative as in Theorem~\ref{thm:maximal_existence}. If the trajectory is confined to a compact subset of $\mathcal{D}$ for all time, then $T_{\max} = \infty$, and the global finiteness of the switching degree follows from the fact that an infinite number of switches would imply an accumulation point of intersections with $\mathcal{B}$, contradicting the finiteness property of semi-algebraic curves.
\end{proof}

\begin{theorem}\label{thm:global_existence}
Consider the polynomial differential variational inequality system under the same assumptions as in Theorem~\ref{thm:maximal_existence}. Suppose there exists a compact set $\mathcal{C} \subset \mathcal{D}$ that is positively invariant in the sense that any regular solution starting in $\mathcal{C}$ remains in $\mathcal{C}$ for its entire lifetime. Then for every initial condition $x_0 \in \mathcal{C}$, the maximal regular solution $(x(\cdot), u(\cdot))$ provided by Theorem~\ref{thm:regular_solution_existence} is defined for all $t \geq 0$, and its switching degree over $[0, \infty)$ is finite.
\end{theorem}

\begin{proof}
Let $x_0 \in \mathcal{C}$ and let $(x(\cdot), u(\cdot))$ be the maximal regular solution from Theorem~\ref{thm:regular_solution_existence}, defined on $[0, T_{\max})$. By the positive invariance assumption, $x(t) \in \mathcal{C}$ for all $t \in [0, T_{\max})$. Since $\mathcal{C}$ is compact and $\mathcal{C} \subset \mathcal{D}$, the trajectory cannot leave every compact subset of $\mathcal{D}$ in finite time. Therefore, the blow-up alternative stated in Theorem~\ref{thm:maximal_existence} forces $T_{\max} = \infty$, establishing global existence.

To prove the finiteness of the total switching degree, we note that the entire trajectory $x([0, \infty))$ is contained in the compact semi-algebraic set $\mathcal{C}$. The bifurcation set $\mathcal{B} \cap \mathcal{C}$ is also a semi-algebraic set of dimension at most $n-1$. The image of the trajectory $x([0, \infty))$ is a connected, one-dimensional semi-algebraic subset of $\mathcal{C}$. It is a standard consequence of the Cell Decomposition Theorem for semi-algebraic sets that such a one-dimensional set can intersect a lower-dimensional semi-algebraic set only finitely many times. Hence, the set $\{ t \geq 0 : x(t) \in \mathcal{B} \}$ is finite, which means the control $u(\cdot)$ has only finitely many discontinuities on $[0, \infty)$. This completes the proof.
\end{proof}

\section{The SOS-PDVIRK4 Algorithm}
\label{sec:algorithm}

The analysis in Section~\ref{Existence} shows that, under Assumption~\ref{ass:core}, a regular solution $(x(\cdot), u(\cdot))$ to the PDVI exists. Furthermore, the semi-algebraicity of the equilibrium graph $\mathrm{Graph}(\mathcal{S})$ implies that for almost every state $x$, the solution set $\mathcal{S}(x)$ is finite and admits locally continuous selections. These properties ensure that the static VI subproblem is well-posed with isolated solutions at almost every state $x_k$, making it amenable to exact computation.

Building on this, we present the \textbf{SOS-PDVIRK4} algorithm, a high-order numerical method for solving PDVIs as defined in Definition~\ref{def:pdvi}. 
The core idea of SOS-PDVIRK4 is to decouple the differential and algebraic components with an explicit fourth-order Runge--Kutta (RK4) scheme advances the state dynamics, while the Moment-SOS hierarchy handles the nonconvex, multivalued equilibrium map $x \mapsto \mathrm{SOL}(K, F(x, \cdot))$. At each time step, Nie’s finite convergence framework for polynomial variational inequalities~\cite{Nie2025} under our assumption guarantees exact recovery of all isolated KKT points, and hence all solutions to $\mathrm{VI}(x_k)$, in finitely many steps. This enables the algorithm to either select a specific branch (e.g., the one closest to the previous iterate) or explore multiple solution manifolds, yielding a robust and theoretically sound procedure for PDVIs.

\subsection{Computation of Equilibrium Controls via Nie's Framework}

Following the gap minimization paradigm~\cite{Nie2025}, solving $\mathrm{VI}(x)$ is equivalent to finding $u \in K$ such that the gap function
\[
\gamma(u) := \min_{z \in K} \langle F(x, u),\, z - u \rangle
\]
satisfies $\gamma(u) = 0$. Note that $\gamma(u) \leq 0$ for all $u \in K$, and equality holds if and only if $u$ solves $\mathrm{VI}(x)$. In particular, if $u$ is not a solution, then $\gamma(u) < 0$.  The key insight of Nie’s approach is that, when the KKT system of (VI$(x)$) has finitely many solutions, the set of all isolated VI solutions can be recovered exactly by solving a sequence of polynomial optimization problems (POPs) via sufficiently high-order Moment-SOS relaxations.

The numerical scheme requires solving two related problems at each parameter value $x$: (i) computing a single equilibrium $u \in \mathrm{SOL}(K, F(x,\cdot))$ for integration, and (ii) enumerating the full finite set $\mathcal{U}_{\mathrm{iso}}(x) = \mathrm{SOL}_{\mathrm{iso}}(K, F(x,\cdot))$ to enable branch selection under multiplicity.

Both tasks are realized through a cutting-plane strategy combined with objective perturbation, as formalized in Algorithms~\ref{alg:single_vi} and~\ref{alg:all_vi}.

\begin{algorithm}[htbp]
\caption{Compute a Single Solution of ($\mathrm{VI}(x)$)}
\label{alg:single_vi}
\begin{algorithmic}[1]
\Require State $x$, polynomial mapping $F(x; u)$, constraint polynomials $\{g_i\}_{i=1}^p$, relaxation order $d$
\Ensure A point $u^* \in \mathrm{SOL}(K,F)$ with gap $\gamma^* = 0$

\State Generate a random positive definite matrix $Q \succ 0$ and set $f(u) := [1; u]^\top Q [1; u]$
\State Initialize empty cutting-plane set $\mathcal{C} \gets \emptyset$
\Repeat
    \State Solve the Moment-SOS relaxation of order $d$ for: 
    \Statex $\displaystyle \min_{u} f(u) \quad \text{s.t. } \begin{cases} g_i(u) \geq 0, & i=1,\dots,p, \\ c(u) \geq 0, & \forall c \in \mathcal{C} \end{cases}$
    \State Extract candidate $u_{\text{cand}}$ from moment solution
    \State Compute gap: $\displaystyle \gamma = \min_{z \in K} (z - u_{\text{cand}})^\top F(x; u_{\text{cand}})$
    \If{$\gamma = 0$}
        \State \Return $u^* \gets u_{\text{cand}}$, $\gamma^* \gets \gamma$
    \EndIf
    \State Let $z^* = \operatorname{arg\,min}_{z \in X} (z - u_{\text{cand}})^\top F(x; u_{\text{cand}})$
    \State Add cutting plane: $\mathcal{C} \gets \mathcal{C} \cup \{ (z^* - u)^\top F(x; u) \}$
\Until{maximum iterations reached}
\State \Return $u^* \gets u_{\text{cand}}$, $\gamma^* \gets \gamma$ \Comment{fallback}
\end{algorithmic}
\end{algorithm}

\begin{algorithm}[htbp]
\caption{Enumerate All Isolated Solutions of ($\mathrm{VI}(x)$)}
\label{alg:all_vi}
\begin{algorithmic}[1]
\Require Same as Algorithm~\ref{alg:single_vi}, max solutions $N_{\max}$, exclusion tolerance $\delta_0 > 0$, distinctness threshold $\tau = 10^{-6}$
\Ensure Set $\mathcal{Y} = \{ u^{(1)}, \dots, u^{(N)} \}$ of distinct isolated solutions

\State Obtain $u^{(1)}$ via Algorithm~\ref{alg:single_vi}; set $f_{\text{prev}} \gets f(u^{(1)})$
\State $\mathcal{Y} \gets \{ u^{(1)} \}$, $\delta \gets \delta_0$
\For{$j = 2$ to $N_{\max}$}
    \Repeat
        \State Solve ($\mathrm{VI}(x)$) with additional constraint $f(u) \geq f_{\text{prev}} + \delta$
        \If{no solution found}
            \State $\delta \gets \delta / 2$
        \EndIf
    \Until{solution $u_{\text{new}}$ found \textbf{or} $\delta < 10^{-6}$}
    \If{$\delta < 10^{-6}$}
        \State \Return $\mathcal{Y}$
    \EndIf
    \If{$\| u_{\text{new}} - u^{(i)} \| > \tau$ for all $u^{(i)} \in \mathcal{Y}$}
        \State $\mathcal{Y} \gets \mathcal{Y} \cup \{ u_{\text{new}} \}$
        \State $f_{\text{prev}} \gets f(u_{\text{new}})$
    \Else
        \State $\delta \gets \delta / 2$ \Comment{Avoid duplicate; reduce gap}
    \EndIf
\EndFor
\State \Return $\mathcal{Y}$
\end{algorithmic}
\end{algorithm}

Algorithm~\ref{alg:single_vi} iteratively refines a candidate solution by adding linearizations of the VI condition as cutting planes, ensuring eventual feasibility. Algorithm~\ref{alg:all_vi} builds upon it by enforcing objective-level exclusion to discover new solutions, leveraging the fact that distinct isolated solutions attain distinct values under a generic quadratic objective $f(u)$. Both algorithms terminate with exact solutions when the SOS relaxation order $d$ is sufficiently large, as guaranteed by Nie’s finite convergence theorem~\cite[Thm.~3.3]{nie2013finite}.

\subsection{Time-Stepping with Adaptive Hybrid VI Solving}

We consider a dynamical system in which the state $x(t) \in \mathbb{R}^n$ evolves according to $\dot{x} = f(x, u)$, and the control $u \in \mathbb{R}^m$ is constrained to belong to the solution set of a parametric polynomial variational inequality 
$S(x) := \{ u \in K \mid \langle F(u; x), z - u \rangle \geq 0,\ \forall z \in K \}$.
Here, $K \subset \mathbb{R}^m$ is a basic closed semialgebraic set, and $F(\cdot; x)$ is polynomial in both $u$ and the parameter $x \in \mathbb{R}^n$.

By Theorem~\ref{thm:structural_properties_general}, the discontinuity set of the solution mapping,
$\Sigma := \{ x \in \mathbb{R}^p \mid S \text{ is not continuous at } x \}$,
is a proper semialgebraic subset of lower dimension. Consequently, along any smooth trajectory $x(t)$, the equilibrium path $u(t) \in S(x(t))$ is generically smooth but may undergo \emph{bifurcations}, \emph{merging}, or \emph{disappearance} precisely when crossing $\Sigma$. These events typically occur near the boundary $\partial G(K)$ of the feasible image or at degenerate constraint configurations.

To robustly simulate the non-smooth dynamics induced by structural changes in the solution set $\mathcal{S}(x)$, we propose \textbf{SOS-PDVIRK4}: a fourth-order Runge–Kutta integrator coupled with an adaptive hybrid solver for the polynomial variational inequality (PVI). The method is designed to respect the underlying semialgebraic geometry of $\mathcal{S}(x)$ and to resolve transitions across the discontinuity set $\Sigma = \partial G(K)$.

At each Runge–Kutta stage, all isolated solutions of the PVI are enumerated via the Moment–SOS hierarchy. The solver employs a regime-dependent formulation:
in the \emph{generic region} ($x \notin \Sigma$), where the linear independence constraint qualification (LICQ) holds, we apply the Lagrange Multiplier Elimination (LME) method~\cite{Nie2025}, which reduces the KKT system to a pure polynomial system in $u$, enabling efficient solution enumeration at low SOS relaxation orders.
In contrast, when $x$ lies near $\Sigma$ where LICQ may fail or interior and boundary solutions coexist the LME formulation becomes ill-conditioned. In this \emph{singular region}, we switch to a full primal-dual KKT formulation that retains both state and multiplier variables, ensuring numerical well-posedness at degenerate points.

Proximity to $\Sigma$ is detected either analytically, when an explicit semialgebraic description of $\Sigma$ is available (e.g., via quantifier elimination for low-degree problems), or numerically, by monitoring abrupt changes in the cardinality or norm of $\mathcal{S}(x(t))$ over recent time steps. If $\operatorname{dist}(x, \Sigma) < \delta$ for a prescribed tolerance $\delta > 0$, the step size $h$ is reduced by a factor of $10$ to prevent overshooting bifurcation events and to resolve emerging solution branches accurately.

Finally, to select a single-valued control trajectory from the finite set $\mathcal{Y} = \mathcal{S}(x(t))$, we enforce temporal continuity via the selection rule
\begin{equation}
u_{\text{new}} = \arg\min_{u \in \mathcal{Y}} \| u - u_{\text{prev}} \|^2,
\tag{SEL}
\end{equation}
which picks the solution branch closest to the previous iterate. This minimal-perturbation criterion ensures compatibility with dynamical stability and avoids spurious jumps at bifurcation points.

The RK4 scheme is applied with equilibrium consistency which at each of the four stages, the intermediate state is computed, the corresponding PVI is solved using the hybrid strategy, and the control is selected via (SEL). The complete simulation procedure is summarized in Algorithm~\ref{alg:sos-rk4}.

\begin{algorithm}[htbp]
\caption{SOS-PDVIRK4 with Adaptive Hybrid VI Solver}
\label{alg:sos-rk4}
\begin{algorithmic}[1]
\Require 
    Initial state $x_0$, time horizon $T$, nominal step size $h_0$, 
    SOS relaxation order $d$, boundary threshold $\delta > 0$
\Ensure 
    Set of valid trajectories $\{(x_k^{(j)}, u_k^{(j)})\}$ for each initial branch $j$

\State $\mathcal{u}_0 \gets \Call{EnumerateAllVI}{x_0, d}$
\If{$\mathcal{Y}_0 = \emptyset$} \Return $\emptyset$ \EndIf

\State $\mathcal{T} \gets \emptyset$
\For{each $u_0^{(j)} \in \mathcal{Y}_0$}
    \State $x \gets x_0$, \ $u \gets u_0^{(j)}$, \ $t \gets 0$, \ $h \gets h_0$, \ $\texttt{success} \gets \texttt{true}$
    \While{$t < T$ and \texttt{success}}
        \State $h \gets \min(h, T - t)$
        
        \Statex \textbf{// Detect proximity to $\Sigma$}
        \If{\Call{IsNearBoundary}{$x, \delta$}}
            \State $h \gets h / 10$ \Comment{Refine near bifurcations}
        \EndIf
        
        \Statex \textbf{// RK4 stages with hybrid VI solving and (SEL)}
        \State $x_1 \gets x$, \quad $\mathcal{Y}_1 \gets \Call{EnumerateAllVI}{x_1, d}$, \quad \textbf{if} $\mathcal{Y}_1 = \emptyset$ \textbf{then break}
        \State $u_1 \gets \arg\min_{u \in \mathcal{Y}_1} \| u - u \|^2$
        
        \State $x_2 \gets x + \tfrac{h}{2} f(x_1, u_1)$, \quad $\mathcal{Y}_2 \gets \Call{EnumerateAllVI}{x_2, d}$, \quad \textbf{if} $\mathcal{Y}_2 = \emptyset$ \textbf{then break}
        \State $u_2 \gets \arg\min_{u \in \mathcal{Y}_2} \| u - u_1 \|^2$
        
        \State $x_3 \gets x + \tfrac{h}{2} f(x_2, u_2)$, \quad $\mathcal{Y}_3 \gets \Call{EnumerateAllVI}{x_3, d}$, \quad \textbf{if} $\mathcal{Y}_3 = \emptyset$ \textbf{then break}
        \State $u_3 \gets \arg\min_{u \in \mathcal{Y}_3} \| u - u_2 \|^2$
        
        \State $x_4 \gets x + h f(x_3, u_3)$, \quad $\mathcal{Y}_4 \gets \Call{EnumerateAllVI}{x_4, d}$, \quad \textbf{if} $\mathcal{Y}_4 = \emptyset$ \textbf{then break}
        \State $u_4 \gets \arg\min_{u \in \mathcal{Y}_4} \| u - u_3 \|^2$
        
        \Statex \textbf{// Update state}
        \State $x_{\text{next}} \gets x + \tfrac{h}{6} \big( f(x_1,u_1) + 2f(x_2,u_2) + 2f(x_3,u_3) + f(x_4,u_4) \big)$
        \State $\mathcal{Y}_{\text{next}} \gets \Call{EnumerateAllVI}{x_{\text{next}}, d}$, \quad \textbf{if} $\mathcal{Y}_{\text{next}} = \emptyset$ \textbf{then break}
        \State $u_{\text{next}} \gets \arg\min_{u \in \mathcal{Y}_{\text{next}}} \| u - u_4 \|^2$
        
        \State Store $(x, u)$; \ $x \gets x_{\text{next}}$; \ $u \gets u_{\text{next}}$; \ $t \gets t + h$
    \EndWhile
    
    \If{\texttt{success}} Append trajectory to $\mathcal{T}$ \EndIf
\EndFor
\State \Return $\mathcal{T}$
\end{algorithmic}
\end{algorithm}

Each call to $\textsc{EnumerateAllVI}$ implements the adaptive hybrid strategy described above. Under generic conditions (finitely many KKT points), Nie’s finite convergence theorem~\cite[Thm.~3.3]{nie2013finite} guarantees exact recovery of all isolated solutions when the SOS order $d$ is sufficiently large. Thus, the selection rule (SEL) operates on the true solution set, ensuring mathematical fidelity of the simulated trajectory even in the presence of complex bifurcations induced by the non-smooth structure of $S(x)$.

\begin{remark}[Initialization on $\Sigma$]
If $x_0 \in \Sigma$ (e.g., on $\partial G(K)$ with degenerate active constraints), $S(x_0)$ may be infinite or form a manifold. Our framework assumes $x_0 \notin \Sigma$ or that such degeneracies are regularized a priori. Handling fully degenerate equilibria remains beyond the scope of the current finite-solution paradigm.
\end{remark}

\subsection{Convergence Analysis}
\label{subsec:convergence}

Our \texttt{SOS-PDVIRK4} algorithm provides rigorous guarantees on both \emph{feasibility} and \emph{numerical accuracy} that the computed $u_k$ is always a feasible VI solution, and the state trajectory $\{x_k\}$ converges to the true Carath\'{e}odory solution with a quantifiable error bound.

The convergence analysis relies on the state trajectory remaining within the parameter feasibility set $\mathcal{D}$, where the polynomial variational inequality's solution mapping $\mathcal{S}(x)$ possesses a well-behaved structure that enables exact recovery via the Moment--SOS hierarchy. Local fourth-order convergence is established under the following assumptions on the exact solution $(x^*(\cdot), u^*(\cdot))$.

\begin{enumerate}[label=\textup{(B\arabic*)}, widest=(B4), left=0pt, align=left]
    \item \textbf{Trajectory confinement}. $x^*(t) \in \mathcal{D}$ for all $t \in [0,T]$.

    \item \textbf{Manifold structure}. There exists an open neighborhood $\mathcal{N} \supset x^*([0,T])$ such that for every $x \in \mathcal{N} \cap \mathcal{D}$, the linear independence constraint qualification holds at each $u \in \mathcal{S}(x)$. Consequently, $\mathrm{Graph}(\mathcal{S}) \cap ((\mathcal{N} \cap \mathcal{D}) \times \mathbb{R}^m)$ is a finite union of disjoint $C^1$ manifolds.

    \item \textbf{Isolated equilibria and uniform SOS exactness}. For each $x \in \mathcal{N} \cap \mathcal{D}$, the set $\mathcal{S}(x)$ is finite. By Nie's theorem~\cite[Thm.~3.3]{nie2013finite}, there exists a uniform relaxation order $d$ such that the Moment--SOS hierarchy of order $d$ recovers $\mathcal{S}(x)$ exactly for all $x \in \mathcal{N} \cap \mathcal{D}$.

    \item \textbf{Consistent branch selection}. The algorithm employs the temporal continuity rule $u_k = \arg\min_{u \in \mathcal{S}(x_k)} \|u - u_{k-1}\|$. Under (B2)--(B3), there exists $h_0 > 0$ such that for any $h < h_0$, this rule uniquely selects a sequence $\{(x_k, u_k)\}$ lying on the same manifold component as the true solution $\{(x^*(t_k), u^*(t_k))\}$.
\end{enumerate}

Under (B2)--(B4), the discrete-time closed-loop vector field $f_{\mathrm{cl}}^h(x) := f(x, u_h(x))$ is continuously differentiable in a neighborhood of the true trajectory for sufficiently small step size $h$. This permits the application of classical ordinary differential equation convergence theory.

\begin{theorem}[Local Fourth-Order Convergence]
\label{thm:convergence}
Suppose (B1)--(B4) hold. Then, there exist constants $h_0 > 0$ and $C > 0$ such that for any step size $h \in (0, h_0)$, the numerical solution $\{x_k\}$ from Algorithm~\ref{alg:sos-rk4} satisfies
\[
\max_{0 \leq k \leq N} \|x_k - x^*(t_k)\| \leq C h^4,
\quad t_k = kh.
\]
\end{theorem}

\begin{proof}
By (B1)--(B2), the exact trajectory lies on a single $C^1$ manifold $\mathcal{M} \subset \mathrm{Graph}(\mathcal{S})$. From (B3), for sufficiently small $h$, all RK4 stage predictors stay within $\mathcal{N} \cap \mathcal{D}$, and the SOS subproblem recovers $\mathcal{S}(\tilde{x})$ exactly. Assumption (B4) ensures the continuity rule selects the unique point in $\mathcal{S}(x_k) \cap \mathcal{M}$, making $u_h(x)$ a $C^1$ function along the numerical path. Thus, the algorithm implements the classical fourth-order Runge--Kutta method on a smooth $C^1$ system, yielding the global $\mathcal{O}(h^4)$ error bound~\cite{hairer1993solving}.
\end{proof}

This convergence is uniform over any compact subset of a single stratum in the semialgebraic partition from Theorem~\ref{thm:structural_properties_general}. Crucially, the framework is fail-safe: if a prediction leaves $\mathcal{D}$, the SOS subproblem either becomes infeasible or exhibits a duality gap, providing a computable certificate of failure.

\begin{remark}[On Asymptotic SOS Convergence]
If (B3) fails, one may use the asymptotic convergence of Lasserre’s hierarchy~\cite{Lasserre2001}. The error then includes a controllable bias $\varepsilon(r) \to 0$ as the relaxation order $r \to \infty$, with fourth-order temporal convergence holding up to this bias.
\end{remark}

\section{Numerical Experiment}
\label{sec:numerics}

We now demonstrate the capabilities of SOS-PDVIRK4 on a physically motivated polynomial differential variational inequality (PDVI) arising from a cyber-physical control system with \emph{nonconvex, annular actuation constraints}. This example captures key challenges in modern engineering applications such as fluidic actuators with minimum flow requirements and maximum power limits, or electromagnetic systems operating within a bounded but non-idling operational envelope where the admissible control commands are confined to a nonconvex set and the associated equilibrium mapping is inherently multivalued.

\subsection{Model Derivation: A Cyber-Physical System with Annular Actuation Constraints}
\label{subsec:model_derivation}

We consider a cyber-physical control system consisting of a second-order nonlinear plant and an actuation subsystem subject to nonconvex physical constraints. The plant state is denoted by $x = (x_1, x_2)^\top \in \mathbb{R}^2$, where $x_1$ is a generalized position and $x_2 = \dot{x}_1$ its velocity. The control input $u = (u_1, u_2)^\top \in \mathbb{R}^2$ is generated by a high-level policy but must respect the actuator's operational limits.

The actuators cannot operate at zero power (e.g., due to minimum flow rates or static friction) and are capped by a maximum power capacity. This is modeled by the annular feasible set
\[
K := \left\{ u \in \mathbb{R}^2 : 1 \leq \|u\|^2 \leq 4 \right\},
\]
which is compact, nonconvex, and has a disconnected boundary $\partial K$ comprising two concentric circles. This geometry is a primary source of solution multiplicity in the associated variational inequality.

The coordination between plant state and actuator command is encoded via a polynomial variational inequality (VI). We define the VI operator $F: \mathbb{R}^2 \times \mathbb{R}^2 \to \mathbb{R}^2$ as
\[
F(x, u) = 
\begin{bmatrix}
u_1^3 - x_1 \\
u_1 u_2^2 - x_2
\end{bmatrix}.
\]
The first component models a cubic force–position relationship in the first actuator channel, regulated to track $x_1$. The second captures a cross-channel modulation effect, where the influence on velocity depends on $u_1$ scaled by $u_2^2$, representing gain-scheduling or energy coupling between actuators.

At each time $t \geq 0$, the admissible command $u(t)$ satisfies the VI:
\begin{equation}
    \langle F(x(t), u(t)), z - u(t) \rangle \geq 0, \quad \forall z \in K.
    \label{eq:vi_model}
\end{equation}
This selects a $u(t) \in K$ that locally minimizes the discrepancy encoded by $F(x(t), \cdot)$ over the feasible set.

The plant evolves according to the polynomial dynamics
\begin{equation}
    \dot{x}(t) = f(x(t), u(t)) = 
    \begin{bmatrix}
        x_2(t) - u_1(t) - 2 \\
        -\delta x_2(t) - \alpha x_1(t) - \beta x_1(t)^3 + \gamma_1 u_1(t) x_2(t) + \gamma_2 u_2(t) x_1(t)
    \end{bmatrix},
    \label{eq:dyn_model}
\end{equation}
with parameters $\alpha = 1$, $\beta = 0.5$, $\delta = 0.1$, $\gamma_1 = 0.8$, and $\gamma_2 = 0.6$. The first equation enforces a velocity offset relative to the control input, while the second describes a Duffing-type oscillator with nonlinear stiffness ($\beta x_1^3$) and state-dependent forcing terms modulated by $u$.

Because $F$ is non-monotone in $u$ and $K$ is nonconvex, the VI~\eqref{eq:vi_model} may admit multiple isolated solutions for a given $x$. Consequently, the resulting polynomial differential variational inequality (PDVI) comprising~\eqref{eq:vi_model} and~\eqref{eq:dyn_model} exhibits multi-modal, path-dependent behavior, making it a challenging benchmark for numerical solvers such as the SOS-PDVIRK4 method proposed in this work.

\subsection{Numerical Experiments: Capturing the Full Solution Landscape}

To comprehensively validate our SOS-PDVIRK4 framework, we conduct a series of numerical experiments that demonstrate its three core capabilities: (i) the ability to capture the complete set of solution branches for a polynomial differential variational inequality (PDVI), (ii) robust numerical convergence under mesh refinement, and (iii) precise tracking of the underlying parameterized variational inequality's (PVI) solution structure through monitoring the cardinality of $\mathcal{S}(x(t))$.

We integrate the system over the time interval $[0, T]$ with $T = 1$, using a uniform temporal discretization with step size $h = 10^{-2}$ ($N = 100$ steps). At each time step, the polynomial variational inequality (PVI) is approximated via a sum-of-squares (SOS) relaxation of order $d = 5$. The initial state is fixed at $x_0 = (2,\,1)^\top$.

To capture the full multiplicity of equilibria at $t=0$, we launch three independent simulations, each initialized from a distinct equilibrium control $u_0^{(i)} \in \mathcal{S}(x_0)$, where $\mathcal{S}(x_0)$ denotes the solution set of the PVI at $x_0$. These initial controls are given by
\[
\begin{aligned}
u^{(1)}_0 &\approx (\,0.969244,\ -1.749448\,)^\top, \\
u^{(2)}_0 &\approx (\,1.259921,\ \phantom{-}0.890899\,)^\top, \\
u^{(3)}_0 &\approx (\,1.259921,\ -0.890899\,)^\top,
\end{aligned}
\]
and satisfy the annular constraint $\|u^{(i)}_0\|^2 \in [1,\,4]$, i.e., $u^{(i)}_0 \in K := \{ u \in \mathbb{R}^2 : 1 \leq \|u\|^2 \leq 4 \}$. Each $u^{(i)}_0$ is a numerically verified solution of the PVI at $t = 0$, thereby providing admissible initial conditions for the dynamical evolution.

To accurately resolve the critical transitions where solution branches merge or bifurcatetypically occurring near the boundary $\partial G(K)$ we employ an adaptive strategy. In these sensitive regions, the time step is refined to $h_{\text{fine}} = h/10 = 0.001$ (i.e., $N=1000$). This enhanced resolution allows the algorithm to make more informed selections between coexisting equilibria based on the temporal continuity rule, thereby providing a faithful representation of the system's dynamics at structural boundaries.

Our results are summarized in Figures~\ref{fig:all_branches}, \ref{fig:vi_count}, and~\ref{fig:conv_branch1}, ~\ref{fig:conv_branch2}, \ref{fig:conv_branch3}.

\begin{figure}[htbp]
    \centering
    \includegraphics[width=0.95\textwidth]{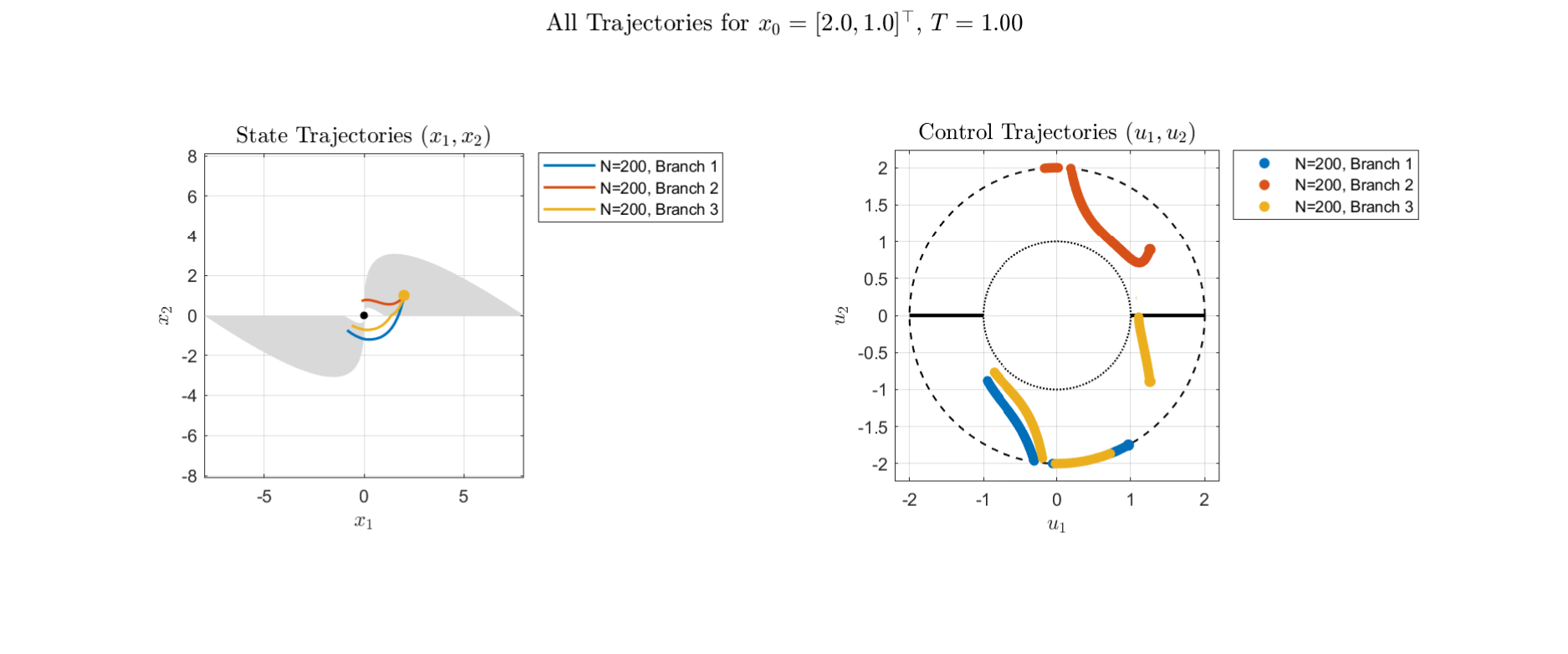}
    \caption{
        Complete solution landscape for the PDVI system with $\Delta t = 0.005$ ($N=200$).
        \textbf{Left:} State trajectories $(x_1(t), x_2(t))$ for all three identified branches.
        \textbf{Right:} Corresponding control trajectories $(u_1(t), u_2(t))$ confined to the annular set $K$. This figure demonstrates the algorithm's capability to capture the full set of admissible dynamical paths from a single initial state $x_0 = (2, 1)^\top$.
    }
    \label{fig:all_branches}
\end{figure}

Figure~\ref{fig:all_branches} provides numerical validation of the theoretical existence results for regular solutions established in Section~\ref{Existence}. Starting from the same initial state $x_0 = (2, 1)^\top$, our algorithm identifies three distinct solution branches, each corresponding to a different initial control selection $u_0^{(i)} \in \mathcal{S}(x_0)$ for $i = 1, 2, 3$.

Each computed trajectory pair $(x^{(i)}(\cdot), u^{(i)}(\cdot))$ constitutes a \textbf{regular solution} in the sense of Definition~\ref{def:switching_degree}. Specifically, the control trajectories $u^{(i)}(\cdot)$ are piecewise continuous with a finite number of discontinuities, confirming that the switching degree $\sigma(x^{(i)}(\cdot), u^{(i)}(\cdot))$ is finite on the compact time interval $[0, T]$. For the displayed time horizon, we observe switching degrees of $\sigma^{(1)} = 1$, $\sigma^{(2)} = 1$, and $\sigma^{(3)} = 3$, respectively.

The switching events occur precisely when the state trajectory $x^{(i)}(t)$ intersects the bifurcation set $\mathcal{B}$, as predicted by Theorem~\ref{thm:regular_solution_existence}(iii). This geometric correspondence between state-space crossings and control discontinuities validates our theoretical characterization of the switching mechanism.

Furthermore, the existence of multiple distinct regular solutions from the same initial state $x_0$ demonstrates the intrinsic multistability of the PDVI system, which arises from the non-singleton nature of the solution mapping $\mathcal{S}(x_0)$. This experimental observation directly supports the constructive proof strategy outlined in Theorem~\ref{thm:regular_solution_existence}, where different initial selections $u_0 \in \mathcal{S}(x_0)$ lead to different piecewise continuous feedback laws and consequently to different dynamical evolutions.

The finiteness of switching events in all computed trajectories provides strong numerical evidence for the semi-algebraic geometric property that the intersection of the one-dimensional semi-algebraic trajectory $x^{(i)}([0,T])$ with the lower-dimensional bifurcation set $\mathcal{B}$ is indeed finite, as guaranteed by the theoretical framework.

\begin{figure}[htbp]
    \centering
    \includegraphics[width=0.6\textwidth]{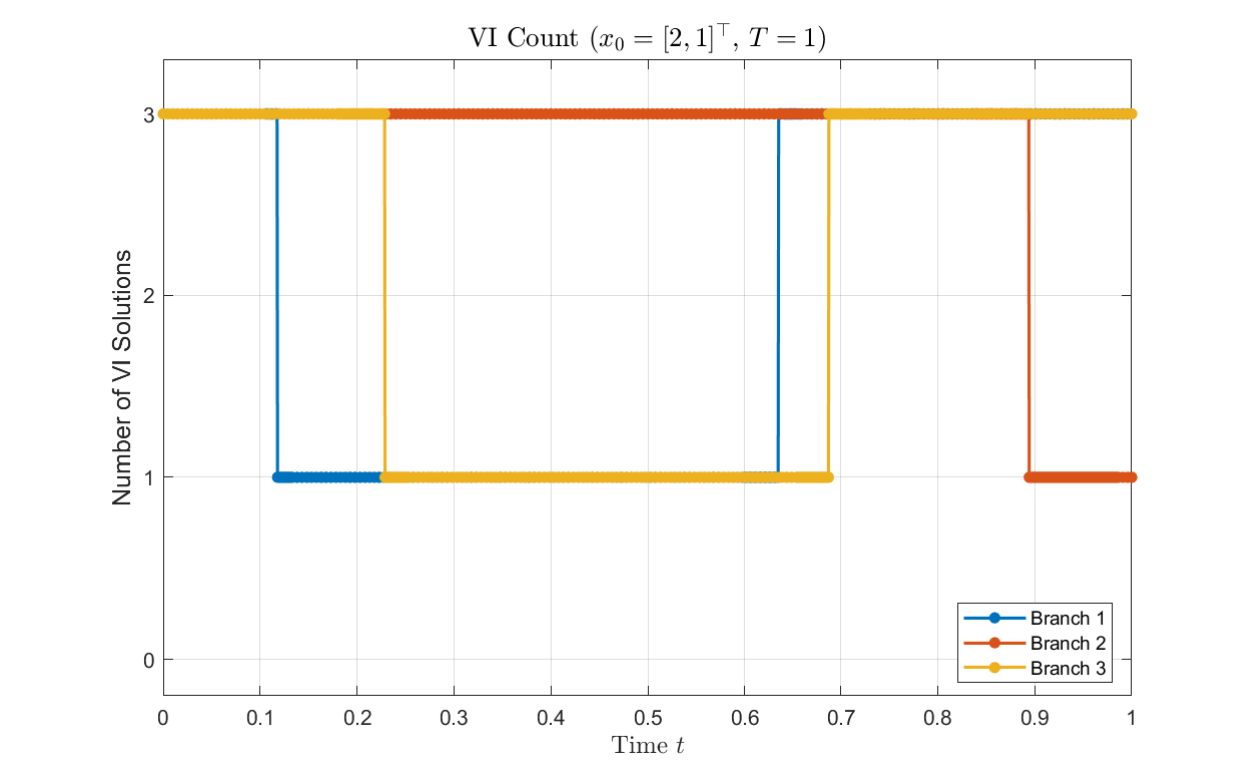}
    \caption{
        Temporal evolution of the cardinality $|\mathcal{S}(x(t))|$ for the parameterized variational inequality (PVI), with $x_0 = (2,1)^\top$, $T=1$, and $N=100$ uniform steps. Three distinct state trajectories corresponding to the three initial equilibrium branches at $t=0$ are tracked simultaneously. 
    }
    \label{fig:vi_count}
\end{figure}

Figure~\ref{fig:vi_count} shows the temporal evolution of the cardinality $|\mathcal{S}(x(t))|$ for three trajectories of the polynomial variational inequality (PVI), all initialized at $x_0 = (2,1)^\top$ with distinct equilibrium controls $u_0^{(i)} \in \mathcal{S}(x_0)$ ($i=1,2,3$). As each state trajectory crosses the boundary $\partial G(K) = \Sigma$, the codimension-one discontinuity set of $\mathcal{S}(\cdot)$ that is the number of isolated solutions changes abruptly.

Specifically:
\begin{itemize}
    \item \textbf{Branch 1} ($u_0^{(1)} \approx (0.969,\,-1.749)^\top$): $|\mathcal{S}| = 3$ on $[0,0.118)$, drops to 1 after exiting $G(K)$, and returns to 3 upon re-entry at $t \approx 0.635$.
    \item \textbf{Branch 2} ($u_0^{(2)} = (1.260,\,0.891)^\top$): remains in $G(K)$ until $t \approx 0.894$, then exits permanently, leaving $|\mathcal{S}| = 1$ thereafter.
    \item \textbf{Branch 3} ($u_0^{(3)} \approx (1.260,\,-0.891)^\top$): exits $G(K)$ at $t \approx 0.229$ ($|\mathcal{S}| = 1$), and re-enters at $t \approx 0.688$, restoring $|\mathcal{S}| = 3$.
\end{itemize}

These transitions align precisely with the theoretical prediction that structural changes in $\mathcal{S}(x)$ occur only across $\Sigma = \partial G(K)$. The \texttt{SOS-PDVIRK4} algorithm robustly tracks all branches and resolves the induced discontinuities without spurious oscillations, thereby certifying the geometric partition of the state space.

\begin{figure}[htbp]
    \centering
    \includegraphics[width=1\textwidth]{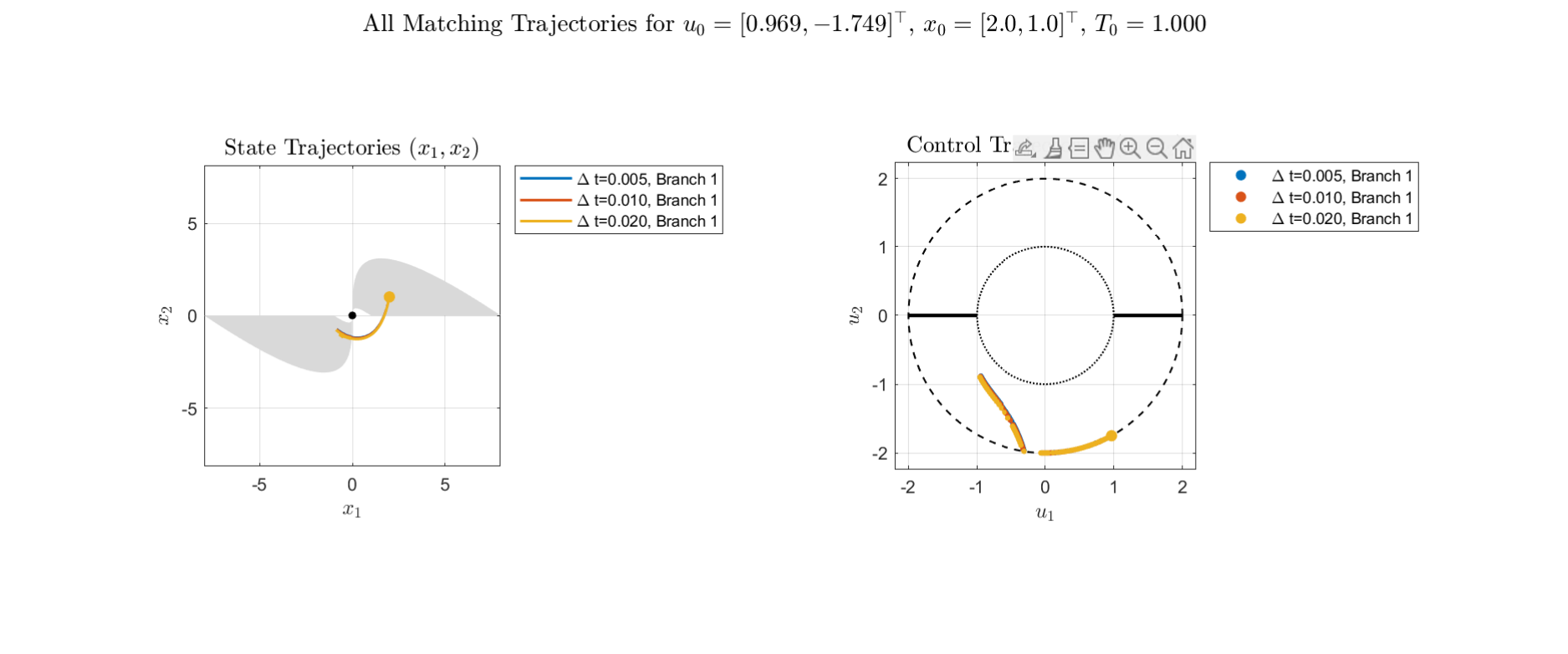}
    \caption{
        Convergence study for the branch initialized at $u_0^{(1)} = (0.969, -1.749)^\top$, comparing trajectories for $\Delta t = 0.02, 0.01, 0.005$.
    }
    \label{fig:conv_branch1}
\end{figure}

\begin{figure}[htbp]
    \centering
    \includegraphics[width=1\textwidth]{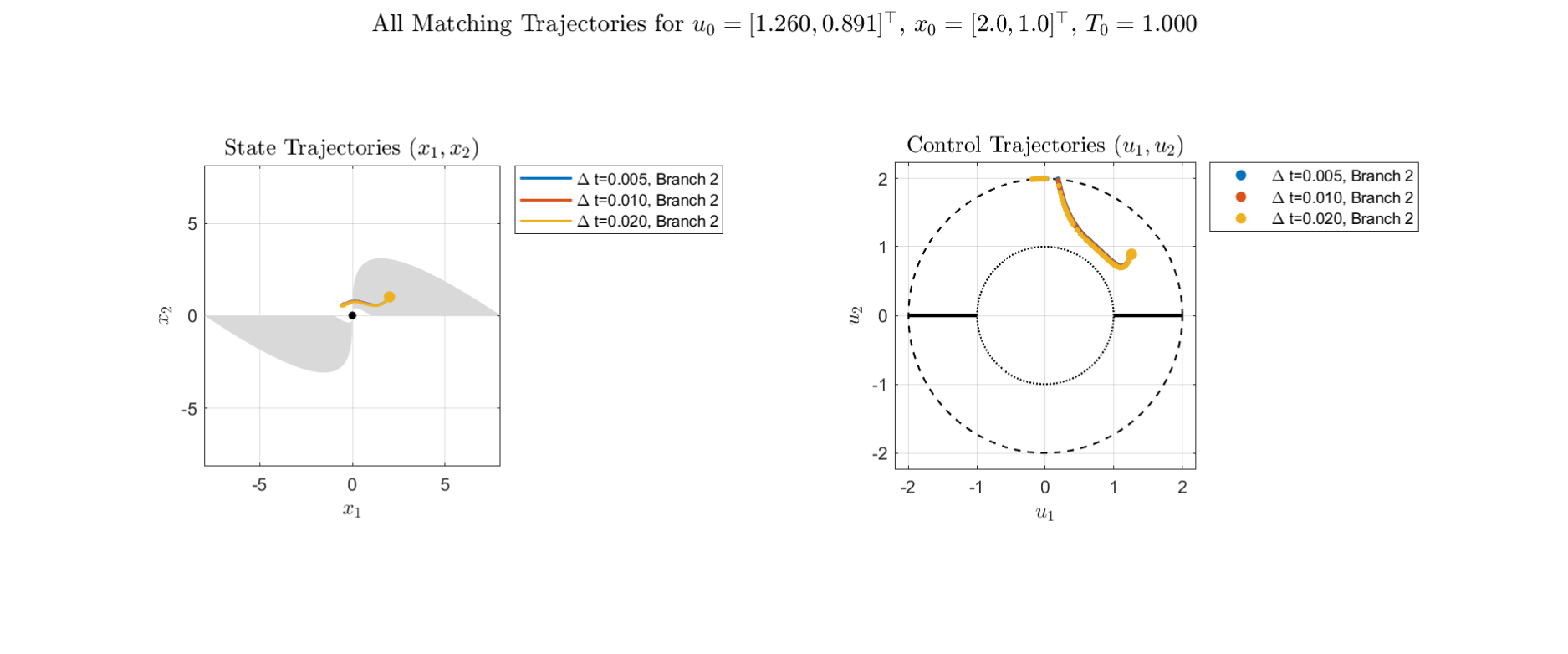}
    \caption{
        Convergence study for the branch initialized at $u_0^{(2)} = (1.260, 0.891)^\top$.
    }
    \label{fig:conv_branch2}
\end{figure}

\begin{figure}[htbp]
    \centering
    \includegraphics[width=1\textwidth]{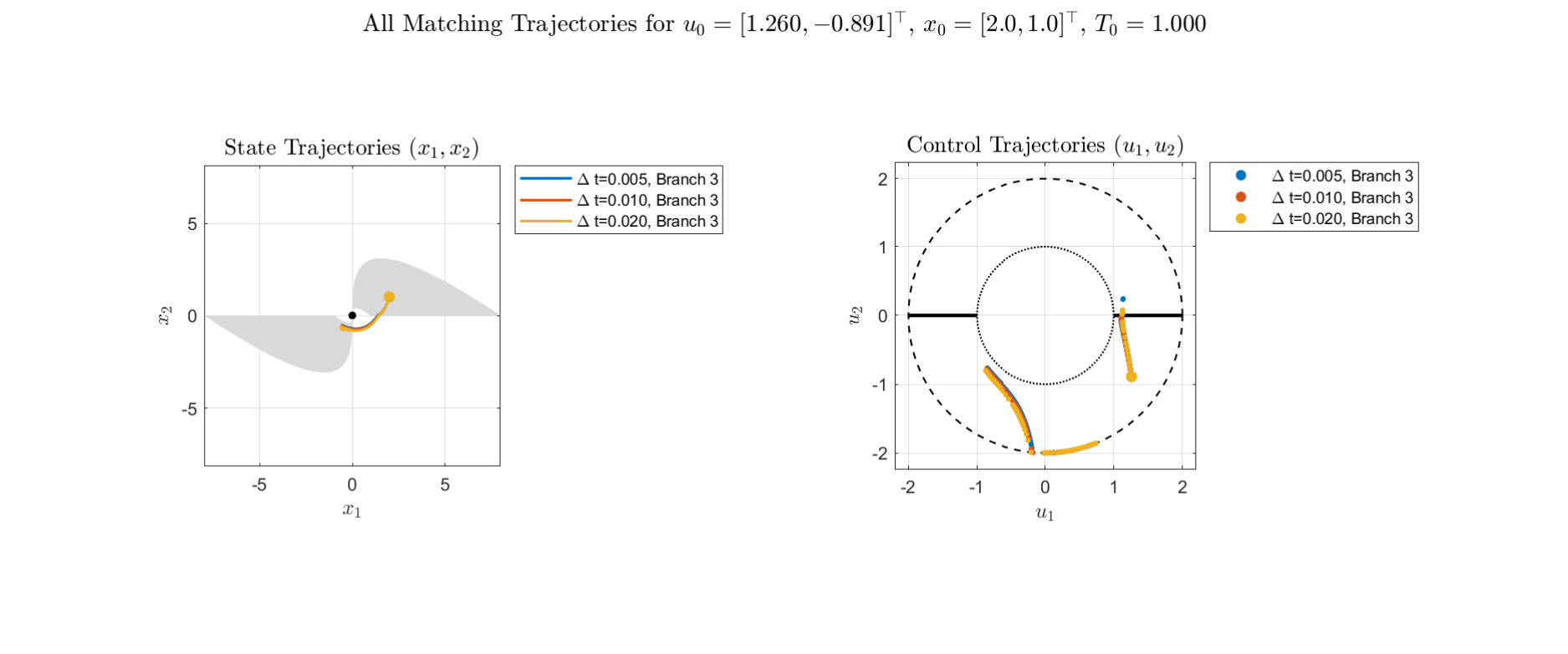}
    \caption{
        Convergence study for the branch initialized at $u_0^{(3)} = (1.260, -0.891)^\top$.
    }
    \label{fig:conv_branch3}
\end{figure}

Figures~\ref{fig:conv_branch1}--\ref{fig:conv_branch3} present convergence studies for all three solution branches, each initialized from a distinct equilibrium $u_0^{(i)} \in \mathcal{S}(x_0)$. For every branch, state and control trajectories are computed with step sizes $\Delta t = 0.02$, $0.01$, and $0.005$. The solutions exhibit clear fourth-order convergence as the mesh is refined, confirming the accuracy and robustness of the \texttt{SOS-PDVIRK4} algorithm across the entire solution manifold.

Moreover, these figures provide direct visual evidence of the structural regularity predicted by Assumptions (B2)–(B4). Specifically, the state trajectories $x(t)$ remain globally continuous, reflecting the smoothness of the underlying Carath\'{e}odory solution. In contrast, the equilibrium controls $u(t)$ exhibit piecewise continuity: they are smooth within each parameter region where the active set of the VI remains constant, but may display kinks or abrupt changes in derivative precisely at the crossing times of the discontinuity set $\Sigma = \partial G(K)$ (e.g., near $t \approx 0.118$, $0.635$, etc.). This behavior aligns with the theoretical characterization of $\mathrm{Graph}(\mathcal{S})$ as a finite union of $C^1$-manifolds, and underscores the algorithm’s ability to faithfully track both the continuous evolution of the state and the piecewise-smooth nature of the equilibrium selection.

\begin{figure}[htbp]
    \centering
    \includegraphics[width=0.75\textwidth]{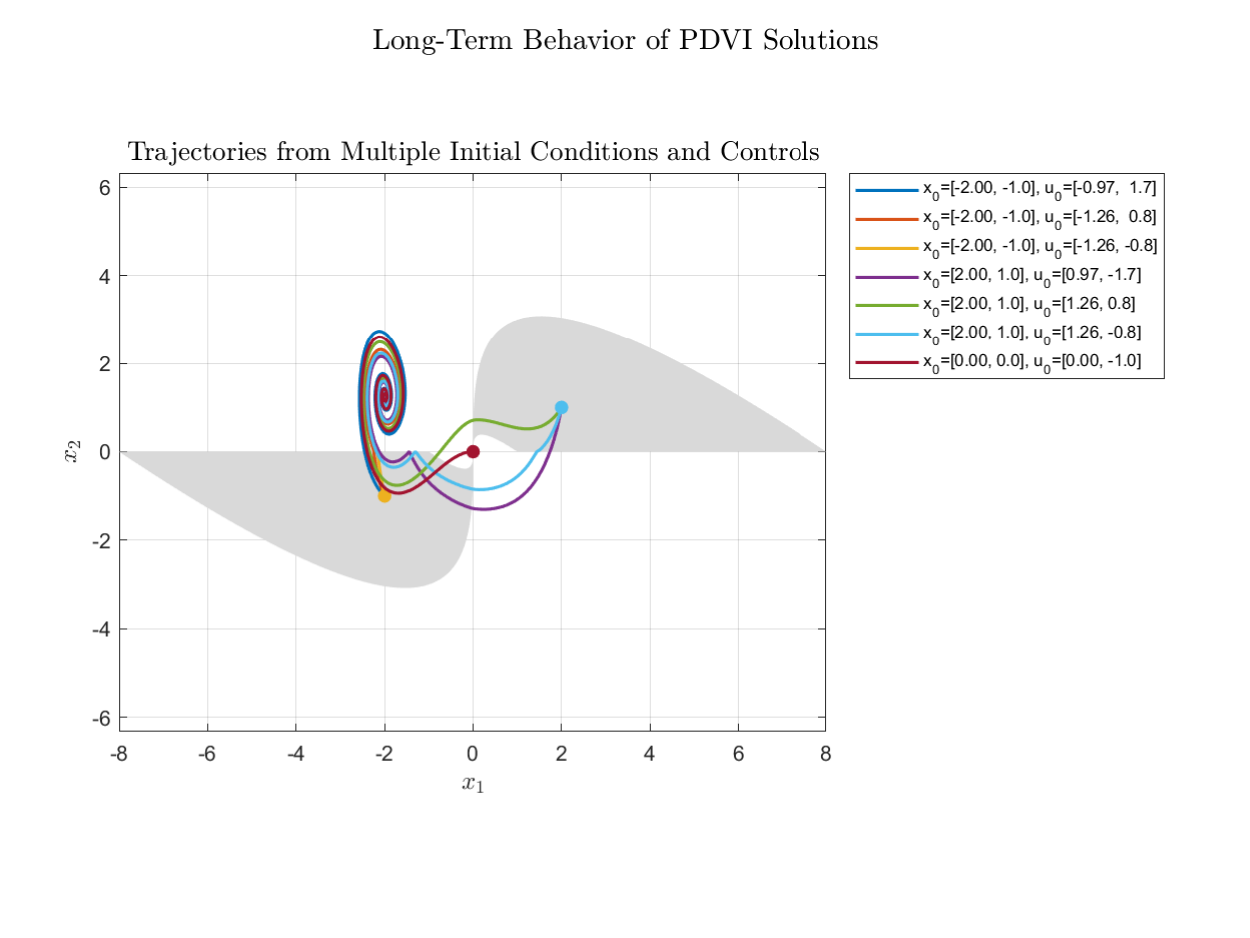}
    \caption[State trajectories for seven initial conditions]{%
        State trajectories $x(t)$ for the PDVI system over $[0, 10]$ with $\Delta t = 0.25$. 
        The seven trajectories correspond to the initial state-control pairs detailed in the main text.
    }
    \label{fig:pdvi_longtime}
\end{figure}

As shown in Figure~\ref{fig:pdvi_longtime}, we simulate the PDVI system from seven initial state-control pairs $(x_0, u_0)$, listed in Table~\ref{tab:initial_conditions}. These include the three equilibria at each of the symmetric states $x_0 = (\pm 2, \pm 1)^\top \in G(K)$, and one admissible control from the origin. Despite this diversity, all trajectories converge as $t \to \infty$ to a common neighborhood near
\[
x^* \approx (-1.9916,\ 1.2409)^\top, \quad u^* \approx (-0.7549,\ 1.8520)^\top,
\]
providing strong numerical evidence for the global attractivity of this equilibrium. 
Crucially, all seven trajectories exhibit only finitely many control discontinuities over the entire horizon $[0,10]$, confirming the global regularity and finite-switching property predicted by Theorem~\ref{thm:regular_solution_existence} and Theorem~\ref{thm:global_existence}.
This confirms the ability of the \texttt{SOS-PDVIRK4} algorithm to robustly capture long-term convergence across the solution manifold.

\begin{table}[ht]
\centering
\caption{Initial state-control pairs used in the long-time simulation of the PDVI system.}
\label{tab:initial_conditions}
\begin{tabular}{c c}
\toprule
$x_0$ & $u_0$ \\
\midrule
$(-2,\ -1)^\top$ & $(-0.97,\ 1.70)^\top$ \\
$(-2,\ -1)^\top$ & $(-1.26,\ 0.80)^\top$ \\
$(-2,\ -1)^\top$ & $(-1.26,\ -0.80)^\top$ \\
$(2,\ 1)^\top$   & $(0.97,\ -1.70)^\top$ \\
$(2,\ 1)^\top$   & $(1.26,\ 0.80)^\top$ \\
$(2,\ 1)^\top$   & $(1.26,\ -0.80)^\top$ \\
$(0,\ 0)^\top$   & $(0.00,\ -1.00)^\top$ \\
\bottomrule
\end{tabular}
\end{table}

In summary, this suite of experiments demonstrates that our framework transcends the role of a numerical solver and functions as a mathematically rigorous analytical instrument. It yields a certified, high-resolution characterization of the full solution manifold of the PDVI explicitly revealing its decomposition into finitely many \textit{regular solution branches}, each distinguished by a well-defined switching sequence $\{t_k\}_{k=1}^{\sigma}$ with finite switching degree $\sigma < \infty$, as guaranteed by Theorem~\ref{thm:regular_solution_existence}. Moreover, the evolution of these branches reflects the underlying semi-algebraic stratification of the parameter domain $\mathcal{D}$, thereby exposing how the algebraic structure of the static VI (e.g., cardinality and continuity of $\mathcal{S}(x)$) governs the instantaneous dynamics.

This geometric insight is empirically validated: trajectories initiated from seven distinct state-control pairs including equilibria at symmetric points and a non-equilibrium configuration exhibit only finitely many control discontinuities over $[0,10]$ (see Figure~\ref{fig:pdvi_longtime}), and all converge asymptotically to the same equilibrium $(x^*, u^*)$. Such uniform long-term behavior, despite initial multistability, confirms the global attractivity predicted under the boundedness condition of Theorem~\ref{thm:global_existence}, and underscores that the solution manifold, while rich in local branching, possesses a coherent global structure.

\section{Conclusion}~\label{sec:conclusion}
This work centers on polynomial differential variational inequalities (PDVIs) a class of nonsmooth dynamical systems in which the instantaneous evolution is governed by a parameter-dependent variational inequality whose data depend polynomially on the state. We begin by analyzing the intrinsic structure of PDVIs, revealing that their solution behavior is dictated by the semi-algebraic stratification of the state-parameter domain $\mathcal{D}$. Within each stratum, the set of equilibria $\mathcal{S}(x)$ exhibits constant cardinality and smooth dependence, while transitions across strata correspond to switching events in the control trajectory.
Building on this geometric insight, we establish that  under the linear independence constraint qualification (LICQ) and mild regularity conditions, Carath\'{e}odory solutions exist for almost every initial condition in $\mathcal{D}$. Moreover, these solutions are \emph{regular} that they admit piecewise continuous controls with only \emph{finitely many switching times} over any compact interval, as formalized in Theorem~\ref{thm:regular_solution_existence}. This finite-switching property ensures that the global solution manifold, though potentially branched, remains combinatorially tractable.

To operationalize these results, we propose \texttt{SOS-PDVIRK4}, a certifiable numerical framework that integrates Nie’s Lagrange multiplier expression (LME) reformulation, the Moment–SOS hierarchy, and a fourth-order Runge–Kutta scheme. At each time step, the algorithm solves a polynomial optimization problem to enumerate all isolated equilibria, thereby capturing transient multiplicity and enabling branch-wise propagation of regular solution trajectories.
We demonstrate the practical relevance of our approach through applications in multi-agent coordination and nonconvex optimal control, where decision-making must account for multiple admissible futures. Numerical experiments confirm two key predictions of our theory: (i) all computed trajectories exhibit only finitely many discontinuities in the control signal, and (ii) despite divergent short-term behaviors, distinct solution branches converge asymptotically to a common equilibrium validating both the \emph{regularity} and \emph{finite-switching} properties of PDVI solutions.

Looking ahead, promising extensions include exploiting sparsity via TSSOS~\cite{Wang2021,Wang2023}, incorporating learning-based heuristics for branch selection, generalizing to stochastic or hybrid PDVI regimes, and developing robust methods for handling boundary-induced degeneracies where LICQ fails. Ultimately, this work advances a vision of \emph{computational stewardship} that by faithfully mapping the landscape of possible evolutions, we empower principled intervention in complex systems shaped by multiplicity, contingency, and choice.

\section*{Acknowledgements}
We thank the anonymous reviewers for insightful comments and the editor for handling
this paper.

\section*{Funding}
This work was supported by the National Natural Science Foundation of China (Grant No.~12301395) and by Natural Science Foundation of Sichuan Province (Grant No.2026NSFSC0236).

\section*{Data Availability} 
Data sharing is not applicable to this article as no datasets were generated or analyzed during the current study.

\section*{Declarations}
\textbf{Conflict of interest} The authors declare that they have no conflict of interest.

\end{document}